\documentclass[11pt,reqno]{amsart}
\usepackage[dvipsnames]{xcolor}
\usepackage{etoolbox,upgreek}
\usepackage{mathrsfs}
\usepackage{amssymb,amsmath,amsthm,color}
\usepackage{graphicx, cite}
\usepackage{url}
\usepackage{enumerate}
\allowdisplaybreaks

\usepackage{mathtools}
\usepackage{url} 
\usepackage{mathtools}
\usepackage[margin=1in]{geometry}
\usepackage{comment}
\usepackage{mathabx}
\usepackage[T1]{fontenc}

\usepackage{footnote}
\usepackage{cite}
\usepackage{amscd}
\usepackage{color}
\usepackage{euscript}
\usepackage{calc}                   

\definecolor{mypink1}{rgb}{0.858, 0.188, 0.478}
\definecolor{mypink2}{RGB}{219, 48, 122}
\definecolor{mypink3}{cmyk}{0, 0.7808, 0.4429, 0.1412}
\definecolor{mygray}{gray}{0.6}

\usepackage[colorlinks, linkcolor=purple, citecolor=LimeGreen, urlcolor=blue, hypertexnames=false]{hyperref}

\usepackage{tikz}
\usepackage{graphicx} 

\newcommand{\loc}{{\rm loc}}

\newtheorem{theorem}{Theorem}[section]
\newtheorem{lemma}[theorem]{Lemma}
\newtheorem{proof of lemma}[theorem]{Proof of Lemma}
\newtheorem{proposition}[theorem]{Proposition}

\theoremstyle{definition}
\newtheorem{definition}[theorem]{Definition}
\newtheorem{assumption}[theorem]{Assumption}
\newtheorem{example}[theorem]{Example}

\newtheorem{remark}[theorem]{Remark}

\numberwithin{equation}{section}

\newcommand{\abs}[1]{\lvert#1\rvert}
\newcommand{\Rn}{\mathbb{R}^n}
\newcommand{\dist}{\mathrm{d}}
\newcommand{\cX}{\mathcal{X}}
\newcommand\norm[1]{\left\lVert#1\right\rVert}

\usepackage{color}

\begin{document}
\title[pointwise convergence to the initial data of heat equations]{A unified framework for  pointwise convergence to the initial data of heat equations in metric measure spaces}

\author{Divyang G. Bhimani, Anup Biswas and Rupak K. Dalai}

\address{Department of Mathematics, Indian Institute of Science Education and Research-Pune, Homi Bhabha Road, Pune 411008, India}

\email{divyang.bhimani@iiserpune.ac.in, anup@iiserpune.ac.in, rupakinmath@gmail.com}




\begin{abstract}
Given a metric measure space $(\cX, d, \mu)$ satisfying the volume doubling condition, we consider a semigroup $\{S_t\}$ and the 
associated heat operator.
We propose general conditions on the heat kernel
so that the solutions of the associated heat equations attain the initial data pointwise. We demonstrate that these conditions are satisfied by a broad class of operators, including the Laplace operators perturbed by a gradient, fractional Laplacian, mixed local-nonlocal operators, Laplacian on Riemannian manifolds, Dunkl Laplacian and many more. In addition, we consider the Laplace operator in $\Rn$ with the Hardy potential and establish a characterization for the pointwise convergence to the initial data. We also prove similar results for the nonhomogeneous equations and showcase an application for the power-type nonlinearities.
\end{abstract}

\maketitle

\tableofcontents

\section{Introduction}
Let \((\cX, d, \mu)\) be a metric measure space with the volume doubling property. Given a metric measure space $(\cX, d, \mu),$
a family \(\{\varphi_t\}_{t > 0}\) of non-negative measurable functions $\varphi_t(x,y)$ on $\cX \times \cX$ is called a \textit{\textbf{heat kernel}} or a \textit{\textbf{transition density}} if the following conditions are satisfied, for all $x, y \in \cX$ and \(s, t > 0\):
\begin{enumerate}
    \item Conservative (stochastic completeness):  $\int_\cX \varphi_t(x, y) \, d\mu(y) = 1.$
\item Semigroup:  $ \varphi_{s+t}(x, y) = \int_\cX\varphi_s(x, z) \varphi_t(z, y) \, d\mu(z).$
\item  Identity approximation: 
   For any \(f \in L^2(\cX)\),  
   \[
   \int_\cX \varphi_t(x, y) f(y) \, d\mu(y) \longrightarrow f(x) \quad \text{in} \  L^2(\cX) \text{ as } t \to 0.
   \]
\end{enumerate}
We say the heat kernel is symmetric if $\varphi_t(x, y)=\varphi_t(y, x)$. 
It is well-known that any symmetric heat kernel  $\{\varphi_t \}$  gives rise to the (heat) semigroup \(\{S_t\}_{t>0}\) where $S_t$ is an operator on $L^2(\cX)$ defined by
\begin{equation*}\label{DB1}
  S_tf(x)=\int_{\cX} \varphi_t (x, y) f(y) \, d\mu(y).  
\end{equation*}
In this article, given a heat  semigroup $\{S_t\}_{t>0}$ and its infinitesimal generator\footnote{It is natural to refer to $\mathcal{L}$ as the \textit{Laplace operator} of the heat kernel $\{\varphi_t\}$.}  $\mathcal{L}$ (see Section \ref{pre}), we consider the abstract Cauchy problem of the following form
\begin{equation}\label{AB1}
\begin{cases}
   \partial_t u(x,t) + \mathcal{L} u(x,t)  =0
   \\ 
   u(x, 0) = f(x)
\end{cases}, \quad   (x, t) \in \cX\times \mathbb R_+.
\end{equation}
Formally, the solution to \eqref{AB1} is given by \[u(x, t)=S_tf(x)= e^{-t \mathcal{L}}f(x)\]
which also known as the mild solution or semigroup solution to \eqref{AB1}. 

It is natural to question under what conditions the above solution converges pointwise to the given initial data.
In fact, in this note, our primary objective is to develop a unified framework for understanding the pointwise convergence behavior of $S_tf$ as time $t\to 0$ when $f$ belongs to  a suitable  weighted Lebesgue spaces.  Specifically, we aim to {\it characterize} the weight class $D_p$ (to be defined below) so that
the  following limit
$$ \lim_{t\to 0}u(x, t)= f(x)$$
holds for almost all $x,$ for  every $f\in L^p_v(\cX)$ and $v\in D_p$. 

In order to state our first main result, we briefly set the notations. Let $v:\cX \to (0, \infty)$ be a positive weight function. In this article, we consider only those weights 
for which $v\, d\mu$ is a Radon measure.
The weighted Lebesgue space norm is given by 
\[\|f\|_{L^p_v(\cX)} = \left(\int_{\cX} |f(x)|^p v(x) \,d\mu (x) \right)^{1/p} \]
with natural modification  for $p=\infty$. The following weight class is important for our main result.
\begin{definition}[Weight class $D_p$]\label{weight}
Let $v$ be a strictly positive weight on  $\cX$,  $1\leq p<\infty$ and $\frac{1}{p}+\frac{1}{p'}=1$. We say $v$ belongs to the class $D_p$  if there exists $t_0\in (0, T)$ such that
	\begin{equation}\label{AB2}
		\norm{ v^{-\frac{1}{p}}\varphi_{t_0}(x, \cdot)}_{L^{p'}(\cX)}<\infty \quad \text{for some}\; x\in\cX.
	\end{equation}
\end{definition}
We define the \textit{\textbf{local (time) maximal operator}} by
\[
H^*_Rf(x) = \sup_{0<t<R}\left|\int_{\cX}\varphi_t(x,y)f(y)\,d\mu(y)\right|\quad \text{ for } R>0.
\]
Throughout this article, we assume that $\{S_t\}$ attains a heat kernel
in some time interval $(0, T)$, and  $[\varphi_t(x, \cdot)]^{-1}$ is bounded on compact sets for every $x\in\cX$. 
To present our result in a very general setting, we impose the following conditions on the heat kernel. These conditions 
are relatable to the one imposed by \cite[Section~5]{BrunoARXIV}.

\begin{assumption}\label{Assump-1.2}
The following hold:
\begin{itemize}
\item[(i)]\label{a1} (convergence on a dense set) For every $f\in C_c(\cX)$, we have
$$\lim_{t\to 0} \int_{\cX} \varphi_t (x, y) f(y) \,d\mu(y)=f(x)$$
for every $x\in\cX$. This is also known as the vague convergence in literature.
\item[(ii)]\label{a2} (parabolic Harnack) There exists $\eta\in (0, 1]$ such that for $R\in (0, T)$ and $x_1, x_2\in \cX$, there exists a constant $C=C(R, x_1, x_2)$ satisfying
$$\varphi_{\eta R}(x_1, y)\leq C\, \varphi_{R}(x_2, y) \quad \text{for}\; y\in\cX.$$
Moreover, for any ball $B_R$ of radius $R$, there exists a positive constant $C=C(x, R)$ satisfying $\varphi_R(x, y)\geq C$ for 
all $y\in B_R$.
\item[(iii)]\label{a3} (time-space maximal inequality) There exist $\gamma\geq 1$, functions
$\xi_1, \xi_2:\cX \times (0, \infty)\to (1, \infty)$ so that
$\xi_i(\cdot, t), i=1,2,$ are bounded on every bounded set for each $t$,
and a function $\Gamma:(0, \infty)\to (0, \infty)$, such that for any $R\in (0, T/\gamma)$ we have
$$H^*_R f(x)\leq \xi_1(x, R) \mathcal{M}_{\Gamma(R)} f(x) + \xi_2(x, R) \int_{\cX} \varphi_{\gamma R}(x, y) f(y) \,d\mu(y)\quad \text{for}\; x\in\cX,$$
for every measurable $f\geq 0$, where $\mathcal{M}_R$ denotes the \textit{\textbf{ local Hardy-Littlewood maximal function}} defined by
$$ 
\mathcal{M}_R f(x)= \sup_{s\in (0, R)}\, \frac{1}{\mu(B(x,r))} \int_{B(x,r)} |f(y)|\, d\mu(y).
$$
\end{itemize}
\end{assumption}
 Condition (ii) above is closely related to the parabolic Harnack inequality that appears in the study of parabolic partial differential equations. For most of the operators, the functions $\xi_1, \xi_2$ in (iii) above would not depend on $R$ only, but allowing its dependence on $x$ provides us extra room to accommodate a larger class of operators. See for instance, the Laplace operator with a Hardy potential in Section~\ref{hp}. We also remark that it is enough to have the conditions in (i) and (iii)
 satisfied $\mu$-almost surely.

Now we can state our first main result.
\begin{theorem}\label{T-1.3} Let $(\cX, d, \mu)$ be a metric measure space with the  volume doubling property.
Let $v$ be a positive weight in $\cX$,\,$\{\varphi_t\}$ be the heat kernel satisfying Assumption~\ref{Assump-1.2} and $1 \leq p<\infty$. 
Then the  following statements are equivalent:
\begin{enumerate}
\item[\hypertarget{1}{(1)}]
 There exists \(R>0\) and a weight \(u\) such that the operator \(H^*_R\) maps \(L_v^p(\cX)\) into \(L_u^p(\cX)\) for \(p > 1,\) and maps \(L_v^1(\cX)\) into weak \(L_u^1(\cX)\) when \(p=1\).
\item[\hypertarget{2}{(2)}]
 There exists \(R>0\) such that \(\int_{\cX}\varphi_R(x,y)f(y)\,d\mu(y)\) is finite for all $x$, and the limit 
 
 \[\lim_{t \rightarrow 0} u(x,t) = f(x)\] 
holds $\mu$-almost everywhere for all \(f \in L_v^p(\cX)\).
\item[\hypertarget{3}{(3)}]
 There exists \(R>0\) such that \(\int_{\cX}\varphi_R(x,y)f(y)\,d\mu(y)\) is finite for some $x$.
\item[\hypertarget{4}{(4)}]
 The weight $v\in D_p.$ 
\end{enumerate}
\end{theorem}

\subsection{Prior work, motivation and new contribution} 
It is well-known that, for $1<p<\infty,$  the  Hardy–Littlewood maximal operator $\mathcal{M}$ is bounded on  weighted  $L^p_w(\mathbb R^n)$ if and only if  $w \in A_p$ (Muckenhoupt weights)\footnote{$w\in A_p$ iff $\left(f_{B}\right)^{p} \leq \frac{c}{w(B)} \int_{B} (f)^{p} w d x $ for $f\geq 0$ and ball $B,$ where $f_B=m(B)^{-1}\int_Bf(x)dx.$ }. See for instance, \cite[Chapter V]{MR1232192}.  As a consequence, we get   ``a.e. pointwise convergence'' for the  classical heat  semigroup $e^{t\Delta}$, that is, for  $f\in L^p_w(\Rn)$, we have
\begin{equation}\label{ds}
  \lim_{t\to 0} e^{t\Delta} f(x) =f(x) \quad \text{for} \  a.e.  \  x \in \Rn. 
\end{equation}
On the other hand, by the abstract Nikishin theory, it is also known that  the a.e.\ pointwise convergence implies  the weak boundedness of  $\mathcal{M}$ form  $L^p_v(\Rn)$  into $L^{p, \infty}_u(\Rn)$ for some weight $u.$ See \cite[Chapter VI]{Jose1985}. However, the strong boundedness  requires a vector-valued approach developed in \cite{Rubio81, Mittag, Carleson81}.   Hartzstein, Torrea and Viviani  in their influential work \cite{vivianiPAMS}  have  successfully adapted this approach to characterize the weight class $D_p$ for  which \eqref{ds} holds. More specifically, they showed that \eqref{ds} holds for $f\in L^p_v(\Rn)$ if and only if  $H^*_R:L^p_v(\Rn)\to L^p_u(\Rn)$ is strongly bounded for some weight $u$ which is again equivalent to $v\in D_p$. A similar result was also proved  for the Poisson equation in \cite{vivianiPAMS}. Since then,  these ``a.e. pointwise convergence" problems have attracted attention of many authors to understand  the  similar behaviour  in various other setting. Specifically, in recent years, ``a.e. pointwise convergence" have already been studied for the following operators:
\begin{itemize}
    \item[-] heat-diffusion problems associated with the harmonic oscillator and the Ornstein-Uhlenbeck opertaor by 
    Abu-Falahah, Stinga and Torrea in \cite{stingaPA}.
    
    \item[-] Laguerre-type operators  by  Garrig\'os et.\ al.\ in \cite{Laggure2014, Lagguare2017}. 
    
    \item[-] Hermite operator $H=-\Delta +|x|^2$ and the associated Poisson case  by  Garrig\'{o}s et.\ al. in  \cite{torreaTAMS}.
    \item[-] Bessel operator by  Cardoso  in \cite{CardosoJEE} 
    \item[-] More recently,  some authors have studied this problem in the non-Euclidean setting. In fact,  Cardoso \cite{CardosoARXIV}, Bruno and Papageorgiou \cite{BrunoARXIV}, and Alvarez-Romero, Barrios and Betancor \cite{ARBB} have  investigated this  problem  on the Heisenberg group, symmetric spaces, and homogeneous trees, respectively, while in \cite{DR1} Bhimani and Dalai have  treated  the torus and waveguide manifold case. 
\end{itemize}
Note that for the Laplacian operator in $\Rn$, the heat kernel at time $t$ is simply a dilation of the heat kernel at time $1$. In this respect, the problem in this article is closely related to topic of almost everywhere convergence for families of convolutions with approximate identities. The latter has already been thoroughly investigated, see for instance \cite{DG23,Kerman,Shapiro,Avram}
and the references therein. However, it is important to note that, in our case, the semigroups are not necessarily generated by convolution or dilation. We should also mention another active program, initiated by the celebrated work of Carleson \cite{Car80},
which analyzes a similar problem of pointwise convergence but for the Schr\"{o}dinger eqaution. For some important development in this
direction, see \cite{DK82,Bo16,DGL17,DZ19} and references therein.

In light of the ongoing interest in this area and the fact that the heat semigroup plays a central role across mathematical analysis, geometry, and probability, we are motivated to propose a general framework for addressing the ``almost everywhere pointwise convergence" problem. This leads us to Assumption \ref{Assump-1.2}. In the process, we also characterize the weighted Lebesgue spaces for the boundedness of maximal operators on metric measure spaces, which may be of independent interest (see Proposition \ref{blmx}), and will play a crucial role in order to prove Theorem \ref{T-1.3}.

Of course, the main mathematical challenge in practice is to verify Assumption \ref{Assump-1.2}.  In this paper, we are able to do this for  the following  wide range of cases:
\begin{enumerate}
    \item  The semigroups arsing in the probability theory, particularly, in the context of L\'evy processes:
  \[e^{-t\psi (-\Delta)} f(x) = \int_{\Rn}e^{-t\psi(\xi)} \hat{u}(\xi)  e^{ix\xi} d\xi\]
 where 
$\psi$ is continuous and strictly increasing satisfying a  certain weak scaling properties. In particular, we cover:
\begin{itemize}
    \item[--] a large class of pure nonlocal operators, including the fractional Laplacian (see Example \ref{Eg2.2})
    \item[--] a large class of local-nonlocal operators, including the sum of Laplacian and fractional Laplacian (see Example \ref{Eg2.5})
\end{itemize}
These operators arise in anomalous transport \cite{Klages}, quantum theory \cite{Daub,LS10}, crystal dislocation \cite{DFV14}, image
reconstruction via denoising \cite{BCM10}, to name just a few. The case of fractional Laplacian is also treated in \cite{BrunoARXIV}
with the help of Caffarelli-Silvestre extension, which may not be possible for general nonlocal operators covered in this article.
Furthermore, our approach also covers the non-tangential convergence for a general family of nonlocal operators (see Remark ~\ref{R-nontan}).
\item  The  Dunkl  operators which to some extent arises from the theory of Riemannian symmetric spaces: see Section \ref{sds} and Theorem \ref{T-Dunkl}.
Generally speaking, these are commuting differential-difference operators, associated to a finite reflection group on the Euclidean space.  They have been successfully applied to the analysis of quantum many body systems of Calogero-Moser-Sutherland type, and have gained considerable interest in mathematical physics,  see \cite{Rosler}. 
\item The  Laplacian with a Hardy potential:
\[\mathcal{L}_b=-\Delta + \frac{b}{|x|^2}.\]
See Section \ref{hp} and Theorem \ref{T8.1}. We could also cover gradient perturbation of the Laplacian, See Example \ref{Eg-2.1}.
The study of $\mathcal{L}_b$ is motivated by various fields in physics and mathematics, including combustion theory, the Dirac equation with Coulomb potential, quantum mechanics, and the analysis of perturbations in classical space-time metrics. See \cite{Kalf, JLVazquez}.
\item The  heat kernels arising from the analysis  of \textit{fractals}.  Roughly speaking,  fractals are subsets of $\Rn$ with certain self-similar properties (e.g. \textit{Sierpi\'nski gasket}). Specifically, we cover
\begin{itemize}
    \item[--]Laplacian on unbounded Sierpi\'nski  gasket $M$ in $\mathbb{R}^2$ (see Example \ref{fractal}) 
\end{itemize}
\item The heat kernel arising from the  differential geometry:
\begin{itemize}
    \item[--] Laplace-Beltrami operator on  Riemannian manifolds (see Example \ref{Eg9.6}) 
\end{itemize}
\end{enumerate}
Next, in Section~\ref{S-nonhom}, we also consider nonhomogenneous problem associated to \eqref{AB1} of the form:
\begin{equation}\label{NLAB1}
\begin{cases}
  \partial_t u + \mathcal{L} u  =F(x, t)
  \\ 
  u(x, 0) = f(x)
\end{cases}, \quad (x, t)\in \cX \times \mathbb R_+.
\end{equation}
Formally, by Duhamel's principle, the solution
of \eqref{NLAB1} is written as
\[u(x, t) = e^{-t\mathcal{L}}f(x)+ \varphi_t\odot F(x, t)=e^{-t\mathcal{L}}f(x)+ \int_0^t \int_{\cX}\varphi_{t-s}(x, y) F(y, s)\, d\mu(y)\, ds.\]
In the early 1980s and 1990s, Weissler, Brezis, Haraux, and Cazenave developed the well-posedness theory for the classical heat equation (i.e., \eqref{NLAB1} with  $\mathcal{L}=-\Delta$ and $F(x,t)= |u|^{\alpha-1} u$ or $u^{\alpha}$ (power type nonlinearity)) on $\Rn$ in their seminal works \cite{brezis1996nonlinear, haraux1982non, Weissler81}. Since then, significant progress has been made in understanding the well-posedness theory for equation \eqref{NLAB1} in various settings over the past four decades. It is impossible to cite all of these developments here, so we refer the reader to \cite[Chapter II]{Philippe} and the references therein. At this point, it is natural to study the ``a.e.\ pointwise convergence problem" for \eqref{NLAB1}. However, there appear to be no results available concerning the pointwise behaviour of solutions to \eqref{NLAB1} even  for the classical heat equation  on $\Rn.$ 

To address this  problem,  we first  characterize the  class of weight  pair $(v,w)$, where    $v:\cX\to (0, \infty)$ and $w:\mathbb{R}_+\to (0, \infty)$  and  formulate  certain assumption on $F$ (see Assumption \ref{Assump-9.1} which is  similar to Assumption \ref{Assump-1.2} above) such that  the following limit 
\[\lim_{t \rightarrow 0} \varphi_t\odot F(x, t) = 0 \  \ a.e. \  x\in \cX \quad  \] 
holds for all  $ F \in L^q_w((0, R), L_v^p(\cX))$
where $vw \in D_{q,p}$, see  Definition \ref{wnp} and  Theorem  \ref{T-9.1}. As a consequence, we get 
\[ \lim_{t\to 0} u(x,t)=  \lim_{t\to 0}\left(e^{-t\mathcal{L}}f(x)+ \varphi_t\odot F(x, t)\right) = f(x)\  \  \text{for} \ \  \ a.e. \  x\in \cX\]
for all \( f \in L_{\tilde v}^p(\mathcal{X}) \) with \( \tilde{v} \in D_p \).   Moreover,   we   showed that  Assumption \ref{Assump-9.1} is met for the  classical semi-linear   heat equation:
\begin{equation}\label{slh}
   \begin{cases}
       \partial_tu-\Delta u = |u|^{\alpha-1}u\\
       u(x, 0)=f(x)
   \end{cases}, \quad (x, t)\in \Rn\times (0,\infty), \alpha>1. 
\end{equation}
Thus,   the solution to \eqref{slh}  converges to $f(x)$ for  a.e.  $x$ and $f\in L^p_{\tilde v}(\Rn), \tilde{v}\in D_p$, under certain condition on $p$ (See Theorem \ref{T9.4} for exact statement). This complements the existing well-posedness theory, and  to the best of the authors' knowledge, this is the first result that  describes the pointwise behavior of non-linear  heat equation  \eqref{slh}.
Of course, it would be interesting to explore similar results for various other nonlinear heat equations. While we will not pursue this here, we believe that our proof method could be useful to readers interested in this direction.

\subsection{Proof techniques and novelties}
We shall see that convergence on a dense set and parabolic Harnack property (Assumption \ref{Assump-1.2} (i) and (ii)) 
together with the boundedness of local (time) maximal operator  $H_R^*$ implies a.e.\ convergence.  In order to prove that $v\in D_p$ implies  $H_R^*$ strongly bounded, we  shall require  time-space maximal inequality (Assumption \ref{Assump-1.2} (iii)). 
To this end, the key ingredient is  to invoke strong boundedness of $\mathcal{M}_R:L^p_v \to L^p_u$ for  $v\in D_p^{\mathcal{M}}\supset D_p$ (see Proposition \ref{blmx}). This essentially rely on vector-valued  weak type  $(1,1)$ inequalities  developed by  Fefferman and Stein \cite{Fefferman},   Grafakos, Liu and Yang  \cite{Grafakos}.   It should be noted that due to  Assumption \ref{AB1} (iii) we can compare  $H_R^*$ and the Hardy-Littlewood maximal function $\mathcal{M}_R$. 

Now, we briefly discuss the key ideas and main  difficulties    for implementing Assumptions \ref{Assump-1.2} in our setup:
\begin{itemize}
    \item In the case of the gradient perturbation of 
$\Delta$ in Example \ref{Eg-2.1} or the Laplace-Beltrami operator
on Riemannian manifolds in Example~\ref{Eg9.6}, the heat kernel enjoys a Gaussian type bound with different Gaussian constants for the upper and lower bounds. In this case, it is possible to follow the idea of 
\cite{vivianiPAMS} with suitable modification. However, in case of purely nonlocal operators, as discussed in Example \ref{Eg2.2}, decay bounds for the heat kernel are no longer Gaussian. Therefore, establishing that the time maximal operator associated with the heat semigroup is dominated by the classical local maximal operator, as stated in Assumption \ref{Assump-1.2} (iii), requires a careful and appropriate choice of annuli (see Lemma \ref{L2.3}). 
 \item For mixed local-nonlocal operators, as considered in Example~\ref{Eg2.5},  sharp lower and upper bounds in the whole space seem to be unknown. The difficulty arises mainly due to the interplay between local and nonlocal operators.
    Consequently, the verification of Assumption~\ref{Assump-1.2}
    becomes far more involved compared to Example~\ref{Eg2.2} (see Lemma~\ref{EL4.6}).
    
    \item In the Dunkl operator framework, although the underlying space is a Euclidean space, the associated measure $\mu_\mathsf{k}$, which depends on the multiplicity function $\mathsf{k}$, is not invariant under standard translations. 
    As a result, the time maximal operator corresponding to the Dunkl heat kernel does not seem compatible with the classical Hardy-Littlewood maximal function. Furthermore, the sharp heat kernel estimate in this case \cite{DA23} involves a distance function
    induced by the reflection group which is not quite favourable for our analysis.
    To address this, we replace the classical Hardy-Littlewood maximal function with the Dunkl maximal function (see \eqref{Dunml-LMF}),  and then employ an analogous approach in proving the boundedness of the local Dunkl maximal functions (see Theorem \ref{T-Dunkl}).
    
    \item When the Hardy potential (inverse square potential) is added to the standard Laplacian, the problem becomes notably more challenging due to the singularities in the heat kernel’s bound at the origin (see \eqref{H-bound}). Moreover, the heat kernel is also not conservative. In this context, establishing pointwise convergence for compactly supported initial data is nontrivial, as addressed in Lemma \ref{L8.2} via Duhamel’s principle.
    Furthermore, the presence of singular terms in the heat kernel’s bound requires a modification of the local maximal function (see \eqref{H-LMF}), which is then shown to be bounded in Lemma \ref{L8.4}. This modification also introduces two locally bounded functions $\xi_1$ and $\xi_2$, in the bound of the local maximal function associated with the heat kernel (see Lemma \ref{L8.3}) which has to be managed appropriately.
    \item When the operator in \eqref{AB1} is considered with a nonhomogeneous term, the domain of the inhomogeneity expands into a time-dependent space. This requires a careful adjustment of the weighted space and the local maximal function, along with a thorough analysis of their boundedness (see Assumption \ref{Assump-9.1} and Theorem \ref{T-9.1}). Notably, we present an interesting application that falls within this framework (see Theorem \ref{T9.4}).
\end{itemize}

The rest of the paper is organized as follows. In Section \ref{pre}, we set the notation and definitions for metric measure spaces and discuss the boundedness of the local maximal function on weighted Lebesgue spaces. Additionally, we recall some fundamental results that will be crucial in proving our main theorems. Section \ref{proof} is devoted to the proof of Theorem \ref{T-1.3}. In Section \ref{examples}, we present examples that satisfy Assumptions \ref{Assump-1.2}, covering the following cases: gradient perturbations of \( \Delta \), a class of nonlocal operators, a class of mixed local-nonlocal operators, the Laplace-Beltrami operator on Riemannian manifolds, the Laplacian on an unbounded Sierpi\'nski gasket \( M \) in \( \mathbb{R}^2 \), Dunkl operators, and the Laplacian with Hardy potential. Finally, in Section \ref{S-nonhom}, we analyze the nonhomogeneous problems associated with \eqref{AB1}, followed by an interesting application.

\section{Preliminaries}\label{pre} 
\subsection{Notations and Definitions}

Let \((\cX, d)\) be a locally compact, separable metric space equipped with a Radon measure \(\mu\). 
Denote the ball with center $x$ and radius $r$ in $(\cX, d)$ by 
\[B(x, r) = \{y \in \cX : d(x, y) < r\}.\]
We assume that  measure $\mu$ on a metric space $\cX$ is volume doubling, that is, the measure of a ball is comparable to the measure of a ball with the same center but half the radius. Specifically,  there is a constant  $C>0$ such that
\[
\mu(B(x, 2r)) \leq C \mu(B(x, r))
\]
for all $x\in \cX$ and  $r>0.$ In this case, we referred  the triple \((\cX, d, \mu)\)  a \textit{\textbf{space of homogeneous type}} or simply a \textit{\textbf{metric measure space}} with doubling property. 
We denote $0$ as some fixed point in $\cX$.
We also define the  space of {\it locally} $L^p-$integrable functions in $(\cX, d, \mu)$ by 
 \[L^p_{\loc}(\cX)=  \left\{ f:\cX \to \mathbb{R}\ \ \text{such that} \; {f\chi_{B(0,r)} \in L^p(\cX)} \quad  \forall\;  r>0 \right\}. \]

 The weak \(L^p\)-quasinorm is given by
\[
\|f\|_{L^{p, \infty}} = \sup_{\lambda > 0} \, \lambda \, \mu\big(\{x \in \mathcal{X} : |f(x)| > \lambda\}\big)^{\frac{1}{p}}.
\]
\subsection{Two weight problem for the Hardy-Littlewood maximal  operator}
For $f\in L^p_{\text{loc}}(\cX),$  $x\in \cX,$ and $R>0,$ the local Hardy-Littlewood  maximal function $\mathcal{M}^{\cX}_R f$ is defined by
$$ 
\mathcal{M}^{\cX}_R f(x)= \sup_{r\in (0, R)}\, \mathcal{A}_rf(x)
$$
where 
\[\mathcal{A}_rf(x)=\frac{1}{\mu(B(x,r))} \int_{B(x,r)} |f(y)|\, d\mu(y).\]
In this subsection, we  prove  that given a weight $v \in D^{\mathcal{M}}_p$(to be defined below),  there exists another weight $u$ so that   $\mathcal{M}^{\cX}_R f:L^p_v(\cX)\to L^p_u(\cX)$ is strongly bounded. 
We define  weight class
\[D_p^{\mathcal{M}}= \left\{ v:\cX\to (0, \infty):v^{-\frac{1}{p}} \in L^{p'}_{\loc}(\cX) \right\}.\]
In what follows, we use the notation $\mu_w$ to  denote the measure induced by weight $w$, that is, $\mu_w(dx)=w(x)\mu(dx)$.

\begin{proposition}[Boundedness of $\mathcal{M}^{\cX}_R$ in metric measure spaces]\label{blmx}
Let $(\cX, d, \mu)$ be a metric measure space with the volume doubling property.
Suppose $v\in D^{\mathcal{M}}_p$. Then there exists a weight \( u \) such that the operator \( \mathcal{M}^{\cX}_R \) maps \( L^p_v(\cX) \) into 
\( L^p_u(\cX) \) for \( p > 1 \), and maps \( L^1_v(\cX) \) into weak \( L^1_u(\cX) \) when \( p = 1 \).

Conversely, for \( 1 \leq p < \infty \), if $\mu_v(B(0,r))<\infty$
for all $r$ and 
the operator \(\mathcal{A}_R\) maps \( L^p_v(\cX) \) into weak \( L^p_u(\cX) \) for some weight \( u \), then \( v^{-\frac{1}{p}} \in L^{p'}_{\loc}(\cX) \).
\end{proposition}

\begin{remark}
The weight class of Proposition \ref{blmx} should be compared with the weight class of Theorem \ref{T-1.3}. In fact, we have $D_p \subset D_{p}^{\mathcal{M}}$, due to Assumption~\ref{Assump-1.2}(ii).
\end{remark}

For the convenience of the reader, we recall the following two useful lemmas first:
\begin{lemma}[Theorem 1.1 in  \cite{TorreaPM}] \label{lm-r1}
Let $(Y, \mu)$ be a measurable space, $\mathcal{B}_1$ and $\mathcal{B}_2$ be Banach spaces. Suppose $T$  be a sublinear operator from $T: \mathcal{B}_1 \rightarrow \mathcal{B}_2$ for some $0<s<p<\infty$, satisfying
$$
\left\|\left(\sum_{j \in \mathbb{Z}}\left\|T f_{j}\right\|_{\mathcal{B}_2}^{p}\right)^{\frac{1}{p}}\right\|_{L^{s}(Y)} \leq C_{p, s}\left(\sum_{j \in \mathbb{Z}}\left\|f_{j}\right\|_{\mathcal{B}_1}^{p}\right)^{\frac{1}{p}},
$$
where $C_{p, s}$ is a constant depending on $Y, \mathcal{B}_1, \mathcal{B}_2,p$ and $s$.
Then there exists a positive function $u$ such that $u^{-1} \in L^{\frac{s}{p}}(Y),\left\|u^{-1}\right\|_{L^{\frac{s}{p}}(Y)} \leq 1$ and
$$
\int_{Y}|T f(x)|^{p} u(x)\, d \mu(x) \leq C_Y\|f\|_{\mathcal{B}_1}^{p}
$$
for some constant $C_Y$ depending on $Y, \mathcal{B}_1, \mathcal{B}_2,p$ and $s$.
\end{lemma}

\begin{lemma}[Kolmogorov inequality, Theorem 3.3.1,  p. 59 in \cite{Guzman}]\label{ki}
Let $\Omega\subseteq \cX$ and $T$ be a sublinear operator maps \(L^p(\Omega, \mu)\) into weak \(L^p(\Omega, \nu)\) for some measures \(\mu, \nu\) and \(1 \leq p < \infty\). Then given $ 0<s<p$, there exists a constant $C$ such that for every subset $A \subset \Omega$ with $u(A):=\int_\Omega \chi_A(x)u(x)\,d\mu(x)<\infty$, we have
\begin{equation*}
\|Tf\|_{L^s(A, \mu)} \leq C \,\mu(A)^{\frac{1}{s}-\frac{1}{p}}\,\|f\|_{L^p(\Omega, \nu)}.
\end{equation*}
\end{lemma}

\begin{proof}[\textbf{Proof of Proposition \ref{blmx}}] Let $0 \in \cX$ be a fixed point and put $E_0=B(0, R)$ and 
$E_k= \{x \in \cX: 2^{k-1}R \leq  d(x,0) <2^k R \}$  for $k\geq 1.$ Then we may write $\cX= \bigcup_{k=0}^{\infty} E_k.$ 
For each $k$ fixed, we split $f=f'+f'',$ where $f'(x)= f(x) \chi_{B(0, R2^{k+1})}(x).$
Note that 
for $x\in E_k$ and $y \notin B(x_0, 2^{k+1}R)$ we have $d(y, x)>R2^k.$ Hence,
\begin{equation*}
   \mathcal{M}_R^{\cX} f''(x)=0 \quad \forall  x \in E_k. 
\end{equation*}
In order to find the desired weight  $u$ on $\cX$,  we first  invoke Lemma \ref{lm-r1} to show that 
\begin{equation}\label{ge}
\|\mathcal{M}^{\cX}_Rf\|_{L^p_{u_k}(E_k)}\leq C_{E_k} \|f\|_{L^p_{v}(\cX)}    
\end{equation}
for some $u_k$ supported on $E_k$ such that $\|u_k^{-1}\|_{L^{s/p}(E_k)} \leq 1$, for $0<s<1<p$.
To this end, the  idea is to  first  apply  the vector-valued weak type \((1,1)\) inequality for $ \mathcal{M}^{\cX}f (\geq \mathcal{M}_R^{\cX}f)$ (see \cite[Theorem 1.2]{Grafakos}) to satisfies the hypothesis of    Kolmogorov inequality (Lemma \ref{ki}). 
In fact, for $0<s<1<p$ and each $k,$ we  have 
\begin{align*}
   \left\|  \left(\sum_{j} |\mathcal{M}^{\cX}_R f_j|^p \right)^{1/p} \right\|_{L^s(E_k)}
   &\leq  \left\|  \left(\sum_{j} |\mathcal{M}^{\cX}_R f'_j|^p \right)^{1/p} \right\|_{L^s(E_k)} \\
    & \leq  C \left(\mu (E_k) \right)^{\frac{1}{s}-1} \left\|  \left(\sum_{j} | f'_j|^p \right)^{1/p} \right\|_{L^1(\cX)} \leq   \tilde{C}_{E_k} \left( \sum_j \|f_j\|^{p}_{L^p_v(\cX)} \right)^{1/p},
\end{align*}
where $\tilde{C}_{E_k}=C \left(\mu (E_k) \right)^{\frac{1}{s}-1} (\int_{E_k} v^{-p'/p} d\mu)^{1/p}. $ This satisfies the hypothesis of  Lemma \ref{lm-r1}.
 We are now ready to define  the weight \(u\) on $\cX:$  
\[
u(x) = \sum_{k=0}^\infty \frac{1}{2^{ k} \tilde{C}_{E_k}} u_k(x) \chi_{E_k}(x).
\]  
 Taking \eqref{ge} into account,  we have 
$\|\mathcal{M}_R^{\cX}f\|_{L^p_u(\cX)} \leq C \|f\|_{L^p_v(\cX)}.$\\
The case $p=1$, essentially follows by invoking the weak type \((1,1)\) inequality for $\mathcal{M}^{\cX}$ on $(\cX, \mu_k, d)$ (see\cite[Theorem 2.1]{MR499948}, \cite[Theorem 3.5]{MR447954}), where $\mu_{k}(d\mu(x))=u_k(x)\, d\mu(x)$ and \( u_k(x) = \chi_{E_k}(x) \).
In fact, we have 
\begin{align*}
\mu_{k}\{x \in E_k : \mathcal{M}^\cX_R f(x) > \lambda\} 
& \leq \mu_{k}\{x \in E_k : \mathcal{M}^\cX_R f^{\prime}(x) > \lambda\}  \leq \mu_{k}\{x \in E_k : \mathcal{M}^\cX f^{\prime}(x) > \lambda\} \\
& \leq \frac{C}{\lambda} \|f^{\prime}\|_{L^1({\cX})}  \leq \frac{\widetilde{C}_k^1}{\lambda} \|f^{\prime}\|_{L^1_v(\cX)}  \leq \frac{\widetilde{C}_k^1}{\lambda} \|f\|_{L^1_v(\cX)},
\end{align*}  
where \(\widetilde{C}_k^1 = C \|v^{-1}(\cdot) \chi_{E_k}(\cdot)\|_{L^\infty(\cX)}\).  
Now taking  weight \( u \) on $\cX$ as  
\[
u(x) = \sum_{k=0}^\infty \frac{1}{2^k \widetilde{C}_k^1} \chi_{E_k}(x),
\]  
 gives the desired inequality.

Conversely, assume that \(\mathcal{A}_R\) maps \( L^p_v(\cX) \) into weak \( L^p_u(\cX) \). Since $\mu_v(B(0, r))<\infty$ for all
$r>0$, we have $\chi_{B(0, r)}\in L^p_v(\cX)$. Then it is easily seen that $\mu_u(B(0, r))<\infty$ for all $r>0$.

Let $x_0\in \cX$ and $R>0$. Then \( B\left(x_0, R/2\right) \subset B(x, R)\subset B(x_0, 2R) \) for \( x \in B\left(x_0, R/2\right) \).
For any nonnegative function \( f \in L^p_v(\cX) \),  we have 
\[
\frac{1}{\mu(B(x_0, 2R))} \int_{B\left(x_0, R/2\right)} f(y) \, d\mu(y)
\leq  \frac{1}{\mu(B(x, R))} \int_{B(x, R)} f(y) \ d\mu(y)=\mathcal{A}_R f(x)\]  for all  \( x \in B(x_0, R/2) \).  Integrating both side (after multiplying $u(x)$)  and then  using Lemma \ref{ki} for \( s < p \), we have  
\begin{align*}
& \mu_u\left(B\left(x_0, R/2\right)\right)^{\frac{1}{s}} \frac{1}{\mu(B(x_0, 2R))} \int_{B\left(x_0, R/2\right)} f(y) \, d\mu(y)\\  
 & \leq \left( \int_{B\left(x_0, R/2\right)} \mathcal{A}_R f(x)^s u(x) \, d\mu(x) \right)^{\frac{1}{s}}\\
 & \leq  \mu_u\left(B\left(x_0, R/2\right)\right)^{\frac{1}{s} - \frac{1}{p}} \left( \int_{\cX} f(x)^p v(x) \, d\mu(x) \right)^{\frac{1}{p}}.
\end{align*}
Putting  \( f = g v^{-\frac{1}{p}} \) where \( g \in L^p(\cX) \) in the above inequality, we get
\[
\int_{B\left(x_0, R/2\right)} g(x) v^{-\frac{1}{p}}(x) \, d\mu(x) \lesssim \mu (B(x_0,2R))\, \mu_u\left(B\left(x_0, R/2\right)\right)^{-\frac{1}{p}} \left( \int_{\cX} g(x)^p \, d\mu(x) \right)^{\frac{1}{p}}.
\]  
By duality, it follows that \( v^{-\frac{1}{p}} \in L^{p'}\left(B\left(x_0, R/2\right)\right) \). Hence, we conclude that \( v^{-\frac{1}{p}} \in L^{p'}_{\text{loc}}(\cX) \), as required.  
\end{proof}

\section{Proof of Theorem \ref{T-1.3}}\label{proof}
\begin{proof}[Proof of Theorem \ref{T-1.3}]
We start by showing \hyperlink{1}{(1)} $\Rightarrow$ \hyperlink{2}{(2)}.
From the Assumption~\ref{Assump-1.2} (i), the limit  
\[
\lim_{t \to 0^+} u(x, t) = f(x)
\]  
holds everywhere for $f\in C_c(\cX)$. Let  $f \in L^p_v(\cX)$ , where \( 1 \leq p < \infty \).    
Define  
\[
\Phi f(x) = \left| \limsup_{t \to 0^+} u(x, t) - \liminf_{t \to 0^+} u(x, t) \right|.
\]  
By our assertion,  if \( f \) is continuous and compactly supported, we have \( \Phi f(x) = 0 \).  
Now, if \( f \in L^p_v(\cX) \), the boundedness of \( H^*_R \) from \( L^p_v(\cX) \) to weak \( L^p_u(\cX) \) implies  
\[
\mu_u \{ x : 2 H^*_R f(x) > \lambda \}^{\frac{1}{p}}  \leq \frac{2C}{\lambda} \| f \|_{L^p_v(\cX)},
\]
where $\mu_u$ denotes the measure with respect to the weight function $u$, that is, $\mu_u(dx)=u(x) d\mu(x)$ and $\lambda>0.$
Again, we have  
\[
\Phi f(x) \leq 2 H^*_R f(x),
\]  
which gives  
\[
\mu_u \{ x : \Phi f(x) > \lambda \}^{\frac{1}{p}}  \leq \frac{2C}{\lambda} \| f \|_{L^p_v(\cX)}.
\]  
Since compactly supported continuous functions are dense in 
$L^p_v(\cX)$ (see \cite[Theorem~13.9]{Ali-Bor}), we can decompose \( f = f_1 + f_2 \), where \( f_1 \) is compactly supported and continuous, and \( \| f_2 \|_{L^p_v(\cX)} \) is arbitrarily small. For such a decomposition, \( \Phi f(x) \leq \Phi f_1(x) + \Phi f_2(x) \), and \( \Phi f_1(x) = 0 \). Therefore,  
\[
\mu_u \{ x : \Phi f(x) > \lambda \}^{\frac{1}{p}}  \leq \frac{2C}{\lambda} \| f_2 \|_{L^p_v(\cX)}.
\]  
Since \( \| f_2 \|_{L^p_v(\cX)} \) can be made arbitrarily small, it follows that \( \Phi f(x) = 0 \) almost everywhere. Consequently,  
\[
\lim_{t \to 0^+} u(x, t)=f(x)
\]  
$\mu$-almost everywhere. Since the operator \( H^*_R f \) maps \( L_v^p(\cX) \) into weak \( L_u^p(\cX) \) for \( p \geq 1 \), there exists some \( x_0 \in \cX \) such that  
$\int_{\cX} \varphi_{R}(x_0, y) |f(y)| \, d\mu(y) < \infty.$
Using Assumption \ref{Assump-1.2} (ii), it follows that
\[
\int_{\cX} \varphi_{\eta R}(x, y) |f(y)| \, d\mu(y) < \infty
\]  
for all \( x \in \cX \). This proves \hyperlink{2}{(2)}.

It is clear that \(\hyperlink{2}{(2)}\) implies \(\hyperlink{3}{(3)}\). Now assume that \hyperlink{3}{(3)} holds. Then, the mapping  
\[
f \mapsto \int_{\cX} \varphi_{R}(x, y) f(y) \, d\mu(y) = \int_{\cX} \varphi_{R}(x, y) v^{-1/p}(y) f(y) v^{1/p}(y) \, d\mu(y)
\]  
is well-defined for every \( f \in L^p_v(\cX) \) and some \( x_0 \in \cX \).  
By the duality principle, the function  
\[
y \mapsto \varphi_{R}(x_0, y) v^{-1/p}(y)
\]  
belongs to \( L^{p'}(\cX) \) for some  \( x_0 \in \cX \). This gives the statement \hyperlink{4}{(4)}.

Assume that \hyperlink{4}{(4)} holds. Therefore, we can find
$x_0\in \cX$ and $t_0\in (0, T)$ such that
$$\norm{ v^{-\frac{1}{p}}\varphi_{t_0}(x_0, \cdot)}_{L^{p'}(\cX)}
<\infty.
$$
Using Assumption~\ref{Assump-1.2} and letting $R=\eta t_0$ we get that
$$
\norm{ v^{-\frac{1}{p}}\varphi_{R}(x, \cdot)}_{L^{p'}(\cX)}
<\infty\quad \text{for all}\; x\in\cX.
$$
By considering the positive and negative parts of \( f \), we may assume \( f \geq 0 \). From Assumption \ref{Assump-1.2} (iii), we have  
\[
H^*_{R_1} f(x) \leq \xi_1(x, R_1) \mathcal{M}_{\Gamma(R_1)} f(x) + 
\xi_2(x, R_1) \int_{\cX} \varphi_{R}(x, y) f(y) \, d\mu(y), \quad \text{for } x \in \cX,
\]  
where $R_1=R/\gamma$, and \( \xi_1, \xi_2 \geq 1 \) are locally bounded functions.  

Since \( v \in D_p \), we have  \( v^{-\frac{1}{p}} \in L_{\text{loc}}^{p'}(\cX) \).    Thus, by Proposition \ref{blmx}, there exists a weight \( u_1 \) such that the local maximal operator $\mathcal{M}_{\Gamma(R_1)}$ maps \( L^p_v(\cX) \) into \( L^p_{u_1}(\cX) \) for \( p > 1 \), and maps \( L^1_v(\cX) \) into weak \( L^1_{u_1}(\cX) \) when \( p = 1 \). 
Again, using H\"{o}lder's inequality, we note that 
\begin{equation}\label{equ-1}
\int_{\cX} \left| \int_{\cX} \varphi_R(x, y) f(y) \, d\mu(y) \right|^p u_2(x) \, d\mu(x)  
\leq \int_{\cX} |f(y)|^p v(y) \, d\mu(y) \cdot \int_{\cX} \|v^{-\frac{1}{p}}(\cdot) \varphi_R(x, \cdot)\|^p_{L^{p'}(\cX)} u_2(x) \, d\mu(x).
\end{equation}
Since  
\[
\|v^{-\frac{1}{p}}(\cdot) \varphi_R(x, \cdot)\|_{L^{p'}(\cX)}
\]  
is finite for all \( x \in \cX \), we can choose $u_2$ suitably to  ensure that the right-hand side of \eqref{equ-1} is finite. In particular, the function
$$ f\to \int_{\cX} \varphi_{R}(x, y) f(y) \, d\mu(y)$$
maps $L^p_v(\cX)$ into $L^p_{u_2}(\cX)$ for $p>1$ and $L^p_v(\cX)$ into
weak $L^p_{u_2}(\cX)$ for $p=1$.

Let $p>1$. Then letting $u=\min\{(\xi_1(\cdot, R_1))^{-p}u_1, 
(\xi_2(\cdot, R_1))^{-p}u_2\}$, we
see that $H^*_{R_1}$ maps $L^p_v(\cX)$ to $L^p_u(\cX)$. Now let $p=1$.  Define
$D_1=\{d(0,x)\leq 1\}$, $D_{i}=\{d(0,x)\leq i\}\setminus D_{i-1}$ for $i\geq 2$. Then $\cX=\cup D_i$.
We define a weight function 
$$w_1(x)=\sum_{k\geq 1} \frac{u_1(x)}{2^k\beta_k} \chi_{D_k},
\quad \text{where} \quad \beta_k=\sup\{\xi_1(x, R_1)\; :\; x\in D_k\}.$$
Then, for $\lambda>0$,
\begin{align*}
\mu_{w_1}\{x\; :\; \xi_1(x, R_1) \mathcal{M}_{\Gamma(R_1)}f(x)>\lambda\}
&\leq \sum_{i} \mu_{w_1}\{x\in D_i\; :\; 
\xi_1(x, R_1) \mathcal{M}_{\Gamma(R_1)}f(x)>\lambda\}
\\
&\leq  \sum_{i} \mu_{w_1}\{x\in D_i\; :\; 
 \mathcal{M}_{\Gamma(R_1)}f(x)>\frac{\lambda}{\beta_i}\}
\\
&\leq  \sum_{i} \frac{1}{\beta_i 2^i}\mu_{u_1}\{x\in D_i\; :\; 
 \mathcal{M}_{\Gamma(R_1)}f(x)>\frac{\lambda}{\beta_i}\}
\\
&\leq \sum_{i} \frac{1}{\lambda 2^i} \norm{f}_{L^1_v(\cX)}.
\end{align*}
Thus, $f\mapsto \xi_1(x) \mathcal{M}_{\Gamma(R_1)}f(x)$ maps
$L^1_v(\cX)$ to weak $L^1_{w_1}(\cX)$. Similarly, we can find a weight 
$w_2$ so that the function
$$ f\to \xi_2(x, R)\int_{\cX} \varphi_{R}(x, y) f(y) \, d\mu(y)$$
maps $L^1_v(\cX)$ into $L^p_{w_2}(\cX)$. Letting $u=w_1\wedge w_2$, 
 we conclude that the operator \( H^*_{R_1} f \) maps \( L^1_v (\cX)\) into weak \( L^1_u (\cX)\). This gives us \hyperlink{1}{(1)}.
\end{proof}

\section{Examples}\label{examples}
In this section, we produce a large family of operators satisfying Assumption~\ref{Assump-1.2}. Our first example 
is a gradient perturbation of the Laplacian.
\begin{example}[gradient perturbation of $\Delta$]\label{Eg-2.1}
Consider a measurable function $b:\Rn\to\Rn$ such that $\abs{b}\in L^p(\Rn)$ for some $p>\max\{n, 2\}$. Define
$$\mathcal{L} u = \Delta u + b(x)\cdot \nabla u.$$
Then there exists a heat kernel $\{\varphi_t\}$ such that for $f\in L^2_{\rm loc}(\Rn)$ the weak solution of \eqref{AB1}
attains the following representation \cite[Theorem~11]{Aronson}
$$ u(x, t)=S_t f(x):=\int_{\Rn} \varphi_t(x, y) f(y)\, dy.$$
Furthermore, for every $T>0$, there exist positive constants $\alpha_1 \geq \alpha_2$ and $ C$ such that
\begin{equation}\label{E2.1}
\frac{t^{-\frac{n}{2}}}{C} e^{-\frac{|x-y|^2}{\alpha_2 t}}\leq \varphi_t(x, y)\leq C t^{-\frac{n}{2}} e^{-\frac{|x-y|^2}{\alpha_1 t}}\quad \text{for all }\; (x, y, t)\in \Rn\times\Rn\times(0, T].
\end{equation}
From the uniqueness of the solution, it is also evident that 
$$\int_{\Rn} \varphi_t(x, y)\, dy=1\quad \text{for all}\; x\in\Rn.$$
Combining \eqref{E2.1} with the above fact it is easily seen that Assumption~\ref{Assump-1.2}(i) holds. Letting $\eta=\frac{\alpha_2}{4\alpha_1}$, we get from \cite[Proof of Prop.~2.1]{vivianiPAMS}
and \eqref{E2.1} that
$$\varphi_{\eta t}(x_1, y)\leq C t^{-\frac{n}{2}} 
e^{-\frac{4|x_1-y|^2}{\alpha_2 t}}\leq C_{x_1, x_2, t}\varphi_t(x_2, y),$$
for $t\leq T$, where the constant $C_{x_1, x_2, t}$ depends only on $x_1, x_2$ and $t$.
Again, since the heat kernel is comparable to the 
fundamental solution of a classical heat operator,
we can easily mimic the proof of \cite[page~1332]{vivianiPAMS} to see that Assumption~\ref{Assump-1.2}(iii) also holds.
See also the proof in Example~\ref{Eg9.6}.
\begin{remark}
 It is possible to extend the set of drifts $b$ to a certain Kato class for which \eqref{E2.1} holds (cf. \cite{Zhang}).   
\end{remark}
\end{example}
\subsection{Heat  semigroups in Probability and Geometry}

Our next example corresponds to pure jump L\'evy processes.
\begin{example}\label{Eg2.2}
Consider an isotropic, unimodal L\'evy process $X=(X_t, \; t\geq 0)$ on $\Rn$. That is, $X$ is a c\`adl\`ag stochastic process with distribution $\mathbb{P}$, $X_0=0$,
the increments of $X$ are independent with rotationally invariant and radially non-increasing density function $p_t(x)$ in $\Rn\setminus\{0\}$. Furthermore,
the following L\'evy-Khintchine formula holds for $\xi\in\Rn$,
$$\mathbb{E}[e^{i\xi\cdot X_t}]=\int_{\Rn} e^{i\xi\cdot y} p_t(y) dy= e^{-t\psi(\xi)},\quad \text{where}\quad
\psi(\xi)=\int_{\Rn} (1-\cos(\xi\cdot x))\, \nu(dx),$$
where $\nu$ is an isotropic unimodal L\'evy measure. Here $\mathbb{E}$ denotes the expectation operator with respect to $\mathbb{P}$. 
The associated semigroup $\{S_t\}$, defined as $S_t f(x)=\mathbb{E}[f(X_t+x)]$, is a convolution semigroup associated to $X$. Furthermore,
$S_t |_{C^\infty_c(\Rn)}$ can be extended to a strongly continuous semigroup on $L^p(\Rn)$ for $p\in [1, \infty)$ (cf. \cite[Chapter 4]{Jacob}).
The infinitesimal generator $\mathcal{L}$ of this semigroup, restricted to $C^\infty_c(\Rn)$, coincides with the pseudo-differential operator $-\psi(D)$ with symbol $\psi$, given by
$$-\psi(D)u(x)=-\int_{\Rn} e^{i\xi\cdot x} \psi(\xi) \widehat{u}(\xi)\, d\xi\quad u\in C^\infty_c(\Rn),$$
where 
$$\widehat{u}(\xi)=\frac{1}{(2\pi)^n}\int_{\Rn} e^{-i\xi\cdot x} u(x) \, dx.$$
In the $L^2$ setting, we can have a Dirichlet form associated to $\psi$ which is also used to define weak solutions for these operators.
From the unimodality of the L\'evy process, we see that $\psi$ is radial. Thus, we let $\psi(\xi)=\psi(\abs{\xi})$ for some function $\psi:[0, \infty)\to [0, \infty)$. We suppose that
$\psi$ is continuous and strictly increasing satisfying the following weak scaling properties: there exist positive numbers $0<\underline{\alpha}\leq \bar\alpha<2$, $\underline{c}\in (0,1]$ and
$\bar{C}\in [1, \infty)$ such that
\begin{equation}\label{E2.2}
\underline{c} \lambda^{\underline{\alpha}}\, \psi(\theta)\leq \psi(\lambda \theta) \leq \bar{C} \lambda^{\bar{\alpha}}\, \psi(\theta)\quad \text{for}\; \lambda\geq 1, \; \theta>0.
\end{equation}
Some standard examples of operators satisfying \eqref{E2.2} are:
\begin{itemize}
\item[(i)] The \textit{\textbf{fractional Laplacian}} where $\psi(\theta)=\theta^{2s}$ for $s\in (0,1)$. In this case, $\underline{\alpha}=\bar{\alpha}=2s$.
\item[(ii)] The \textit{\textbf{mixed fractional operator}} where $\psi(\theta)=\theta^{2s_1} + \theta^{2s_2}$ for $0<s_1<s_2<1$. In this case, $\underline{\alpha}=2s_1$ and $\bar{\alpha}=2s_2$.
\end{itemize}
Many other examples of operators satisfying \eqref{E2.2} can be found in \cite[Section 15]{SSV}, \cite[Section~4.1]{Bogdan}. Under the above hypothesis, it follows
from \cite[Corollary~23]{Bogdan} that
\begin{equation}\label{E2.3}
 \frac{1}{C} \left([\psi^{-}(1/t)]^n\wedge \frac{t\psi(|x|^{-1})}{|x|^n}\right)\leq p_t(x)\leq 
 C \left([\psi^{-}(1/t)]^n\wedge \frac{t\psi(|x|^{-1})}{|x|^n}\right)\quad \text{for}\; x\neq 0, \; t>0,
\end{equation}
for some constant $C$, where $\psi^{-}$ denotes the inverse of $\psi$. In this case, the heat kernel is given by \[\varphi_t(x, y)=p_t(x-y).\]
\end{example}

\begin{lemma}\label{L2.3}
The heat kernel $\{\varphi_t\}$ given by Example~\ref{Eg2.2} satisfies Assumption~\ref{Assump-1.2}.
\end{lemma}

\begin{proof}
Since $p_t$ is a transition density for $t>0$ and its characteristic function converges to $1$, by L\'evy continuity theorem,
we get
$$\lim_{t\to 0} \int_{\Rn} \varphi_t (x, y) f(y) dy=\lim_{t\to 0}\int_{\Rn} f(x-y) p_t(y)dy=f(x)$$
for every bounded continuous function $f$. This gives us Assumption~\ref{Assump-1.2}(i). To see Assumption~\ref{Assump-1.2}(ii),
we take $x_1, x_2\in\Rn$. For any $y$ satisfying $|x_1-y|\leq 2|x_1-x_2|$ we obtain  from \eqref{E2.3} that 
$$\frac{1}{\kappa_R(x_1, x_2)}\leq \varphi_R(x_1, y),~ \varphi_R(x_2, y)\leq \kappa_R(x_1, x_2)$$
for some constant $\kappa_R$. Since $|x_2-y|\leq |x_1-x_2|\Rightarrow |x_1-y|\leq 2|x_2-x_1|$, $\varphi_R(x_1, y)$ and $\varphi_R(x_2, y)$
are comparable. Consider $|x_2-y|\geq |x_1-x_2|\Rightarrow |x_1-y|\leq 2|x_2-y|$. Using \eqref{E2.2} it then follows that
\begin{align*}
 \varphi_R(x_2, y) &\leq C \left([\psi^{-}(1/R)]^n\wedge \frac{R\psi(|x_2-y|^{-1})}{|x_2-y|^n}\right)
 \\
 &\leq C \left([\psi^{-}(1/R)]^n\wedge \frac{2^n R\psi( 2|x_1-y|^{-1})}{|x_1-y|^n}\right)
 \\
 &\leq 2^{n+\bar{\alpha}} C\bar{C} \left([\psi^{-}(1/R)]^n\wedge \frac{R\psi( |x_1-y|^{-1})}{|x_1-y|^n}\right)\leq 2^{n+\bar{\alpha}} C^2\bar{C} \varphi_R(x_1, y). 
 \end{align*}
Thus Assumption~\ref{Assump-1.2}(ii) holds with $\eta=1$.

It remains to verify Assumption~\ref{Assump-1.2}(iii). We define $\phi(r)=\frac{1}{\psi^{-}(r^{-1})}$. By our assumption, $\phi$ is a strictly increasing function.
Fix $R>0$, and let 
$$p_t (x) = p_t(x) \chi_{\{|x|\leq \phi(R)\}} + p_t(x) \chi_{\{|x|> \phi(R)\}}:= W^1_t(x) + W^2_t(x).$$
For $t\in (0, R)$, choose $j\in\mathbb{N}\cup\{0\}$ so that $2^{j} t< R\leq 2^{j+1} t$. Define $r_j=\phi(t2^j)$. Then
\begin{equation}\label{E2.4}
p_t(x) \leq p_t(x) \chi_{\{|x|\leq r_{0}\}} + \sum_{k\leq j+1} p_t(x) \chi_{\{r_{k-1}<|x|\leq r_{k}\}}.
\end{equation}
Since $\psi(\frac{1}{r_k})=\frac{1}{t2^k}$, we obtain from \eqref{E2.3} that
$$p_t(x) \chi_{\{|x|\leq r_{0}\}}\leq C \frac{1}{r^n_0} \chi_{\{|x|\leq r_{0}\}}$$
and
\begin{align*}
 p_t(x) \chi_{\{r_{k-1}<|x|\leq r_{k}\}} &\leq C t \frac{\psi(r^{-1}_k)}{r_k^n}\chi_{\{|x|\leq r_{k+1}\}}
 \\
& = C \frac{r^{n}_{k+1}}{2^k r^{n}_{k}}\frac{1}{r_{k+1}^n}\chi_{\{|x|\leq r_{k+1}\}}
\leq \frac{C}{2^k}\, \left[\sup_{\theta>0}\frac{\psi^{-}(2\theta)}{\psi^{-}(\theta)}\right]^n \frac{1}{r_{k+1}^n}\chi_{\{|x|\leq r_{k+1}\}},
\end{align*}
for some constant $C.$ If we choose $\lambda>1$ so that $\underline{c}\lambda^{\underline{\alpha}}\geq 2$, then applying \eqref{E2.2} we have
$$\psi(\lambda \psi^{-}(\theta))\geq \underline{c}\lambda^{\underline{\alpha}} \theta\geq 2\theta
\Rightarrow \psi^{-}(2\theta)\leq \lambda \psi^{-}(\theta)\quad \text{for all}\; \theta>0.$$
Thus, letting $\Gamma(r)=\phi(2r)$ and combining the above estimate we obtain
$$\sup_{t\leq R} \int_{\Rn} W^1_t(x-y) f(y) dy\leq C_1 \mathcal{M}_{\Gamma(R)} f(x)$$
for all $f\geq 0$, where $C_1= C (1+\lambda^n \sum_{k\geq 1} 2^{-k} )$. On the other hand,
for $r>\phi(R)$ we have
$$R \psi(r^{-1})\leq R \psi\left(\frac{1}{\phi(R)}\right)= R\psi(\psi^{-}(R^{-1}))=1,$$
and $r^{-1}\leq \frac{1}{\phi(R)}=\psi^{-}(R^{-1})$. Thus, for $|x|>\phi(R)$ and $t\leq R$, we have
$$[\psi^{-}(1/t)]^n\geq [\psi^{-}(1/R)]^n\geq r^{-n} \geq \frac{R\psi(r^{-1})}{r^n},$$
giving us
$$\sup_{t\leq R} \int_{\Rn}W^2_t (x-y) f(y)\, dy\leq \int_{\Rn} p_R(x-y) f(y) dy$$
for $f\geq 0$. Hence, Assumption~\ref{Assump-1.2}(iii) holds with $\gamma=1$.
\end{proof}

\begin{remark}
Let $\omega(x)= 1\wedge \frac{\psi(|x|^{-1})}{|x|^n}$. In view of \eqref{E2.3} we see that, in the setting of Example~\ref{Eg2.2},
we have $D_p=\{v>0\; :\; v^{-\frac{1}{p}}\in L^{p'}_\omega(\Rn)\}$.
\end{remark}

\begin{remark}[Non-tangential convergence in nonlocal setting]\label{R-nontan}
The calculation of Lemma~\ref{L2.3} also gives us {\it \textbf{non-tangential}} convergence for the nonlocal operators. Gaussian analogue of this notion can be found in \cite{Non-tangential}
(see also \cite[Section 7.2]{torreaTAMS}).
More precisely,
if we set 
$$\Gamma^p(x)=\{(y, t)\in\Rn\times \mathbb{R}_+\; :\; |x-y|\leq \phi(t)\},$$
which is the nonlocal analogue of parabolic Gaussian cone,
for $\phi(r)=\frac{1}{\psi^{-}(r^{-1})}$, then for some constant $C$, we would have
\begin{equation}\label{nontan}
\sup_{(y, t)\in \Gamma^p(x)} \left|\int_{\Rn} p_t(y-z)f(z) \,dz \right| \leq C \mathcal{M} f(x) \quad \text{for all}\; f\geq 0.
\end{equation}
From \eqref{nontan} it can be easily seen that 
 $$\lim_{(y, t)\in \Gamma^p(x),\, t\to 0} \int_{\Rn} p_t(y-z)f(z)\, dz = f(x) \quad \text{almost surely, for all}\; f\in L^1(\Rn).$$ 
Now to see \eqref{nontan}, we fix $\lambda>1$ so that $\phi(\lambda r)\geq 2\phi(r)$ for all $r>0$. This can be easily checked from
\eqref{E2.2}. Next, for any $t>0$, define $r_j=\phi(\lambda^{j+1} t)$ for $j=0, 1, 2, \ldots$. Write
\begin{equation*}
p_t(y-z)= p_t(y-z)\chi_{\{|x-z|\leq r_0\}} + \sum_{j=1}^\infty p_t(y-z)\chi_{\{r_{j-1}< |x-z|\leq r_j \}}.
\end{equation*}
For $r_{j-1}< |x-z|\leq r_j$, we have 
$$|y-z|\geq |x-z|-|y-x|\geq \phi(\lambda^{j} t) - \phi(t)\geq  \phi(\lambda^{j} t) - \phi(\lambda^{j-1} t)\geq \phi(\lambda^{j-1} t).$$
Thus,
\begin{align*}
p_t(y-z)\chi_{\{r_{j-1}< |x-z|\leq r_j \}}\lesssim \frac{1}{\lambda^{j-1}}\left[\frac{\phi(\lambda^{j+1} t)}{\phi(\lambda^{j-1} t)}\right]^n
\frac{1}{r^n_{j}}\chi_{\{|x-z|\leq r_j \}}
\lesssim \frac{1}{\lambda^{j-1}} \left[\sup_{\theta>0}\frac{\psi^{-}(\lambda^{2}\theta)}{\psi^{-}(\theta)}\right]^n \frac{1}{r^n_{j}}\chi_{\{|x-z|\leq r_j \}}.
\end{align*}
Now following the calculations of Lemma~\ref{L2.3} we get \eqref{nontan}.
\end{remark}

Our next example concerns L\'evy processes with a diffusion part. This corresponds to integro-differential
operators with both local and nonlocal parts.

\begin{example}\label{Eg2.5}
Let $X$ be the L\'evy process given by Example~\ref{Eg2.2} and $(B_t\, , t\geq 0)$ be the Brownian motion, independent of $X$, running 
twice as fast as standard $n$-dimensional Brownian motion. We consider the super-imposed L\'evy process $Y_t=B_t+X_t$. Then 
the corresponding L\'evy symbol is given by $\tilde\psi(\xi)=|\xi|^2+\psi(\xi)$. That is
$$\mathbb{E}[e^{i\xi\cdot Y_t}]=\int_{\Rn} e^{i\xi\cdot y} \tilde{p}_t(y) dy= e^{-t\tilde\psi(\xi)},$$
where $\tilde{p}_t$ denotes the transition density function of $Y_t$. The generator of the corresponding semigroup coincides with
$\Delta-\psi(D)$ on $C^\infty_c(\Rn)$ where $-\psi(D)$ is the pseudo-differential operator given by Example~\ref{Eg2.2}. As before,
the heat kernel is given by $\varphi_t(x, y)=\tilde{p}_t(x-y)$.

Let $\hat{p}_t(x)=\frac{1}{(4\pi t)^{\frac{n}{2}}} e^{-\frac{|x|^2}{4t}}$ be the transition density of $B_t$. Then using the independence of
$X$ and $B$ it follows that 
\begin{equation}\label{E2.5}
\tilde{p}_t(x)=\hat{p}_t\ast p_t(x)=\int_{\Rn} \hat{p}_t(y) p_t(x-y) \, dy=\int_{\Rn} \hat{p}_t(x-y) p_t(y) \, dy,
\end{equation}
where $p_t$ is the transition density of the nonlocal component from Example~\ref{Eg2.2}.
\end{example}

\begin{lemma}\label{EL4.6}
Let $\{\varphi_t\}$ be the heat kernel given by Example~\ref{Eg2.5}. Then $\{\varphi_t\}$ satisfies Assumption~\ref{Assump-1.2}.   
\end{lemma}

\begin{proof}
Following the same argument as in Lemma~\ref{L2.3}, we note that $\{\varphi_t\}$ satisfies Assumption~\ref{Assump-1.2}(i). 
To prove Assumption~\ref{Assump-1.2}(ii), we consider two points $x_1, x_2$. Then for any $R>0$ it can be easily shown that
(see \cite[Proof of Prop.~2.1]{vivianiPAMS})
$$
\hat{p}_{R/4}(x_1-y)\leq \kappa_{x_1, x_2}\, \hat{p}_R (x_2-y)\quad \text{for all}\; y\in\Rn,
$$
where the constant $\kappa_{x_1, x_2}$ depends on $x_1, x_2$ and $R$ only. Again, by \eqref{E2.2} and \eqref{E2.3}, we have
a constant $\kappa$ satisfying
$$p_{R/4}(z)\leq \kappa p_R(z)\quad \text{for all}\; z\neq 0.$$
Therefore, setting $\eta=\frac{1}{4}$ and using \eqref{E2.5}, we obtain
$$\varphi_{\eta R}(x_1, y)\leq \kappa\, \kappa_{x_1, x_2}\, \varphi_R(x_2, y)\quad \text{for all}\; y\in\Rn.$$

Let us now verify Assumption~\ref{Assump-1.2}(iii).
Since 
$\tilde{p}_t$ is radially decreasing (see \cite{Watanabe}), by \cite[Corollary~7]{Bogdan}, we have
\begin{equation}\label{E2.6}
    \tilde{p}_t(x)\leq C t\frac{\tilde{\psi}(|x|^{-1})}{|x|^n}\quad t>0, \; x\neq 0,
\end{equation}
for some constant $C$. Again, since $\hat{p}_t\leq (4\pi t)^{-\frac{n}{2}}$, using \eqref{E2.5} and \eqref{E2.6} we obtain the following bound
\begin{equation}\label{E2.7}
  \tilde{p}_t(x)\leq C \left[t^{-\frac{n}{2}}\wedge t\frac{\tilde{\psi}(|x|^{-1})}{|x|^n}\right]\quad t>0, \; x\neq 0,  
\end{equation}
for some constant $C$.
Define \( \tilde{\phi}(r) = \frac{1}{\tilde{\psi}^{-}(r^{-1})} \). Since \( \tilde{\psi} \) is strictly increasing, \( \tilde{\phi} \) is also strictly increasing. We will now follow the same arguments as in Lemma~\ref{L2.3}. Fix \( R > 0 \), and decompose \( \tilde{p}_t(x) \) as follows  
\[
\tilde{p}_t(x) = \tilde{p}_t(x) \chi_{\{|x| \leq \tilde{\phi}(R)\}} + \tilde{p}_t(x) \chi_{\{|x| > \tilde{\phi}(R)\}} := \tilde{W}^1_t(x) + \tilde{W}^2_t(x).
\]  
For \( t \in (0, R) \), choose \( j \in \mathbb{N} \cup \{0\} \) such that \( 2^j t < R \leq 2^{j+1} t \). Define \( \tilde{r}_j = \tilde{\phi}(t 2^j) \). Then,  
\begin{equation}\label{E2.8}
    \tilde{W}_t(x) \leq \tilde{p}_t(x) \chi_{\{|x| \leq \tilde{r}_0\}} + \sum_{k \leq j+1} \tilde{p}_t(x) \chi_{\{\tilde{r}_{k-1} < |x| \leq \tilde{r}_k\}}.
\end{equation}  
Note that \( \tilde{\psi}\left(\frac{1}{\tilde{r}_k}\right) = \frac{1}{t 2^k} \). We claim that there exists $\kappa_R$, dependent on $R$, such that
$$\tilde\psi(\kappa_R \sqrt{1/t})\leq 1/t\quad \text{for}\; t\in (0, R).$$
Using \eqref{E2.2} we get that
\begin{align*}
 \tilde\psi(\kappa_R \sqrt{1/t})
\leq \kappa_R^2 t^{-1} + \psi(\kappa_R \sqrt{1/t})
&\leq \kappa_R^2 t^{-1} + \bar{C} (R/t)^{\frac{\bar\alpha}{2}} \psi(\kappa_R \sqrt{1/R})
\\
&\leq \kappa_R^2 t^{-1} + \bar{C} (R/t) \psi(\kappa_R \sqrt{1/R})
\\
&\leq \frac{1}{t}(\kappa_R^2 + \frac{1}{R}\psi(\kappa_R \sqrt{1/R})).
\end{align*}
Since $\psi(0)=0$, choosing $\kappa_R$ small enough, we have the claim.
Therefore, using \eqref{E2.7}, we obtain  
\[
\tilde{p}_t(x) \chi_{\{|x| \leq \tilde{r}_0\}} \leq \frac{C}{(\kappa_R)^n} \frac{1}{\tilde{r}_0^n} \chi_{\{|x| \leq \tilde{r}_0\}}
\]
and
\begin{align*}
 \tilde{p}_t(x) \chi_{\{\tilde{r}_{k-1}<|x|\leq \tilde{r}_{k}\}} 
 &\leq C t \frac{\tilde{\psi}(\tilde{r}^{-1}_k)}{\tilde{r}_k^n}\chi_{\{|x|\leq \tilde{r}_{k+1}\}}
 \\
& = C \frac{\tilde{r}^{n}_{k+1}}{2^k \tilde{r}^{n}_{k}}\frac{1}{\tilde{r}_{k+1}^n}\chi_{\{|x|\leq \tilde{r}_{k+1}\}}
\leq \frac{C}{2^k}\, \left[\sup_{\theta>0}\frac{\tilde\psi^{-}(2\theta)}{\tilde{\psi}^{-}(\theta)}\right]^n \frac{1}{\tilde{r}_{k+1}^n}\chi_{\{|x|\leq \tilde{r}_{k+1}\}},
\end{align*}
for some constant $C.$ Note that for \( \lambda > 1 \), applying \eqref{E2.2}, we obtain  
\begin{equation}\label{E1}
    \tilde{\psi}(\lambda \theta) \geq \lambda^2 \theta^2 + \underline{c} \lambda^{\underline{\alpha}}\, \psi(\theta) \geq (\lambda^2 \wedge \underline{c} \lambda^{\underline{\alpha}}) \, \tilde{\psi}(\theta).
\end{equation}  
Now, choose \( \lambda > 1 \) such that \( (\lambda^2 \wedge \underline{c} \lambda^{\underline{\alpha}}) \geq 2 \). Applying \eqref{E1}, we deduce that  
$$
\tilde{\psi}\left(\lambda \tilde{\psi}^{-}(\theta)\right) \geqslant (\lambda^2 \wedge \underline{c} \lambda^{\underline{\alpha}}) \,\theta\geqslant 2 \theta
\Rightarrow\tilde{\psi}^{-}(2 \theta) \leqslant \lambda \tilde{\psi}^{-}(\theta)\quad \text{for all}\; \theta>0.
$$
Thus, letting $\tilde{\Gamma}(r)=\tilde{\phi}(2r)$ and combining the above estimates, we obtain
$$\sup_{t\leq R} \int_{\Rn} \tilde{W}^1_t(x-y) f(y) dy\leq \xi_{1}(R) \, \mathcal{M}_{\tilde{\Gamma}(R)} f(x)$$
for all $f\geq 0$, where $\xi_{1}(R)= C ((\kappa_R)^{-n}+\lambda^n \sum_{k\geq 1} 2^{-k} )$. On the other hand,
for $|x|>\tilde{\phi}(R)$, we have from \eqref{E2.7} that
$$\tilde{p}_t(x)\leq R\frac{\tilde\psi(|x|^{-1})}{|x|^n}.$$
Thus, to estimate $\tilde{W}^2_t$ it is enough to show that there exists a constant 
$\xi_2(R)$ satisfying
\begin{equation}\label{E2.10}
\tilde{p}_R(x)\geq \xi_2(R)\, R  \frac{\tilde\psi(|x|^{-1})}{|x|^n}\quad \text{for}\; |x|\geq \tilde{\phi}(R).
\end{equation}
Once \eqref{E2.10} is established,  it would give us
$$\sup_{t\leq R} \int_{\Rn}\tilde{W}^2_t (x-y) f(y)\, dy\leq \xi_2(R) \int_{\Rn} \tilde{p}_R(x-y) f(y) dy$$
for $f\geq 0$. Hence, Assumption~\ref{Assump-1.2}(iii) holds with $\gamma=1$.

Thus, we remained to prove \eqref{E2.10}. Denote by $\varrho=\tilde{\phi}(R)$. For $|x|\geq \varrho$, we compute using \eqref{E2.5}
that
\begin{align*}
\tilde{p}_R (x)&\geq (4\pi R)^{-n/2}\int_{|y|\leq 2|x|} e^{-\frac{|x-y|^2}{4R}} p_R(y)\, dy
\\
&\geq (4\pi R)^{-n/2} p_R(2|x|) \int_{|y|\leq 2|x|} e^{-\frac{|x-y|^2}{4R}}\, dy
\\
&\geq (4\pi R)^{-n/2} p_R(2|x|) \int_{|y-x|\leq |x|} e^{-\frac{|x-y|^2}{4R}}\, dy
\\
&\geq (4\pi R)^{-n/2} p_R(2|x|) \omega_{n} \int_{0}^{|x|} e^{-\frac{s^2}{4R}} s^{n-1}\, ds
\\
&\geq (4\pi R)^{-n/2} p_R(2|x|) \omega_{n} \int_{0}^{\varrho} e^{-\frac{s^2}{4R}} s^{n-1}\, ds,
\end{align*}
where $\omega_n$ denotes the surface measure of the unit ball, and in the second inequality we used the fact $p_R$ is radially decreasing. From \eqref{E2.3} and scaling of $\psi$ in \eqref{E2.2} we can find a constant $\kappa_R$ so that
$$p_R(2|x|)\geq \kappa_R \frac{\psi(|x|^{-1})}{|x|^n} \quad \text{for}\; |x|\geq \varrho.$$
Again, since
$$|x|^{-2}\leq \varrho^{-2+\bar\alpha} |x|^{-\bar\alpha} \leq \varrho^{-2+\bar\alpha} \bar{C} \frac{\varrho^{-\bar\alpha}}{\psi(\varrho^{-1})}
\psi(|x|^{-1})$$
for some constant $\tilde\kappa_R,$ we have $\psi(|x|^{-1})\geq \tilde\kappa_R\, \tilde\psi(|x|^{-1})$ for all $|x|\geq \varrho$.
Thus we have \eqref{E2.10}.
\end{proof}

\begin{example}\label{Eg9.6}
Let  $(M, \nu)$ be a geodesically complete Riemannian manifold and $\nu$ be the Riemannian volume which also satisfies the volume doubling property. Also, assume that the Ricci curvature of 
$M$ is non-negative. Let $\dist$ denote the geodesic distance on $M$. The Laplace-Beltrami operator $\Delta$ on $M$ generates the heat semigroup $\{S_t\}$, which is associated with a
local Dirichlet form $\mathcal{E}$ given by 
$$\mathcal{E}(f)=\int_M |\nabla f|^2 d\nu, \quad f\in W^{1, 2}_0(M).$$
Here $\nabla f$ denotes the Riemannian gradient. Also, it is well known that $S_t$ has a smooth heat kernel $\varphi_t(x, y)$. Furthermore, by \cite{Li-Yau}, \cite[Section~6.1]{Gri},
\begin{equation}\label{AB101}
\frac{1}{\nu\left(B\left(x, \sqrt{t}\right)\right)} e^{-c_1\frac{\dist(x, y)^2}{t}}\lesssim \varphi_t(x, y)\lesssim \frac{1}{\nu\left(B\left(x, \sqrt{t}\right)\right)} e^{-c_2\frac{\dist(x, y)^2}{t}}\quad \text{for all}\; t>0, \; x, y\in M,
\end{equation}
for some constants $c_1\geq c_2>0$.
Again, by \cite[Example~3.17]{Gri}, we see that $S_t$ is stochastically complete, that is, 
\begin{equation}\label{AB102}
\int_M \varphi_t(x, y)\,\nu(dy)=1\quad \text{for all}\; t>0.
\end{equation}
Using \eqref{AB102}, it can be easily seen that Assumption~\ref{Assump-1.2}(i) holds. Assumption~\ref{Assump-1.2}(ii) can also be easily checked. To verify Assumption~\ref{Assump-1.2}(iii), we recall that
by the volume doubling property of $\nu$, there exists $n>0$ and $C_1\geq 1$ so that 
\begin{equation}\label{AB103}
\nu(B(x, \lambda r))\leq C_1 \lambda^n \nu(B(x, r))\quad \text{for all}\; \lambda\geq 1, \; r>0\; \text{and}\; x\in M.
\end{equation}
Fix $R>0$ and $x\in M$. For $t\in (0, R]$ choose $j\in\mathbb{N}$ so that $2^{j}t\leq R\leq 2^{j+1}t$. As before, choose $\alpha>0$ large enough so that
$$t\mapsto t^{-\frac{n}{2}} e^{-\frac{c_2 r}{t}}\quad \text{is increasing in $(0, R]$ for all}\; r\geq \alpha R.$$
Let $\chi_0=\chi_{\left\{\dist(x, y)\leq \sqrt{\alpha t}\right\}}$ and $\chi_{k}= \chi_{\left\{\sqrt{\alpha 2^{k} t}<\dist(x, y)\leq \sqrt{\alpha 2^{k+1} t}\right\}}$. As before,
we can estimate $$\varphi_t(x, y) \chi_{\left\{\dist(x, y) \leq \sqrt{\alpha R}\right\}} \leq 
 \sum_{k=0}^j \varphi_t(x, y)\chi_k(y)$$ as below
\[\varphi_t(x, y)\chi_0(y) \lesssim \frac{\nu\left(B\left(x,\sqrt{\alpha t}\right)\right)}{\nu\left(B\left(x,\sqrt{ t}\right)\right)} \frac{1}{\nu\left(B\left(x,\sqrt{\alpha t}\right)\right)} \chi_{\left\{\dist(x, y)\leq \sqrt{\alpha t}\right\}}\lesssim \frac{1}{\nu\left(B\left(x,\sqrt{\alpha t}\right)\right)} \chi_{B\left(x, \sqrt{\alpha t}\right)}\]
and 
\begin{align*}
\varphi_t(x, y)\chi_k(y) &\lesssim \frac{\nu\left(B\left(x,\sqrt{\alpha 2^{k+1} t}\right)\right)}{\nu\left(B\left(x, \sqrt{t}\right)\right)} \frac{e^{-c_2\alpha 2^k}}{\nu\left(B\left(x,\sqrt{\alpha 2^{k+1} t}\right)\right)}
\chi_{\left\{\dist(x, y)\leq \sqrt{\alpha 2^{k+1} t}\right\}}
\\
&\lesssim 2^{\frac{nk}{2}} e^{-c_2\alpha 2^k} \frac{1}{\nu\left(B\left(x,\sqrt{\alpha 2^{k+1} t}\right)\right)}  \chi_{B\left(x, \sqrt{\alpha 2^{k+1} t}\right)},
\end{align*}
where we use \eqref{AB103}. Hence,
$$\sup_{t\in(0, R)} \int_{M}\varphi_t(x, y)\chi_{\left\{\dist(x, y)\leq\sqrt{\alpha R}\right\}} f(y)\, d\nu(y) \lesssim \mathcal{M}_{\sqrt{2\alpha R}} f(x),$$
for all $f\geq 0$. Now Assumption~\ref{Assump-1.2}(iii) follows by letting $\gamma=\frac{c_1}{c_2}$. More precisely, since by 
\eqref{AB103}, we have
$$ \nu(x, \sqrt{R}) \leq C_1 (R/t)^{\frac{n}{2}} \nu(x, \sqrt{t}), $$
by our choice of $\alpha$, we obtain for $\dist(x, y)>\sqrt{\alpha R}$ that
\begin{align*}
\varphi_t(x, y)\lesssim \frac{R^{\frac{n}{2}}}{\nu(x, \sqrt{R})}
t^{-\frac{n}{2}} e^{-c_2\frac{\dist(x,y)^2}{t}}
\leq \frac{R^{\frac{n}{2}}}{\nu(x, \sqrt{R})} R^{-\frac{n}{2}} e^{-c_2\frac{\dist(x,y)^2}{R}}\lesssim \varphi_{\gamma R}(x, y).
\end{align*}
\end{example}

Following is an example from the fractals. Many other examples of similar type can be found in \cite{Kigami,Barlow,BB99,FS92}.
\begin{example}\label{fractal}
We consider the unbounded Sierpi\'nski  gasket $M$ in $\mathbb{R}^2$ from \cite{BP88}. The Hausdorff dimension of  $M$ is $d_f=\frac{\log 3}{\log 2}$. Equip $M$ with the
$d_f$-dimensional Hausdorff measure $\nu$. Note that $\nu(B(x, r))\simeq r^{d_f}$ for all $r>0$ and $x\in M$. Also, the gasket metric is equivalent to
the metric induced by the Euclidean norm.
From \cite{BP88} it is also known that there exists a semigroup $\{S_t\}$ associated to the Brownian motion in $M$, which also attains a continuous heat kernel $\{\varphi_t\}$.
Furthermore, for $d_w=\frac{\log 5}{\log 2}$ (the walk dimension) and constants $c_1\geq c_2>0$ we have
$$
t^{-\frac{d_f}{d_w}} \exp\left(-c_1\left(\frac{|x-y|^{d_w}}{t}\right)^{\frac{1}{d_w-1}}\right)
\lesssim \varphi_t(x, y)\lesssim t^{-\frac{d_f}{d_w}} \exp\left(-c_1\left(\frac{|x-y|^{d_w}}{t}\right)^{\frac{1}{d_w-1}}\right)
$$
for all $x, y\in M$ and $t>0$. Setting $\Gamma(r)=(\alpha r)^{\frac{1}{d_w}}$ for a suitable large $\alpha$, we can see that Assumption~\ref{Assump-1.2} holds (see the argument in Example~\ref{Eg9.6}).
\end{example}

\subsection{The Dunkl  operator}\label{sds}
In this section, we show that Assumption~\ref{Assump-1.2} holds with suitable modification for the Dunkl operator and therefore,
an appropriate version of Theorem~\ref{T-1.3} holds for the Dunkl operator as well. To introduce the operator, we consider a normalized root
system $\mathcal{R}\subset\Rn$ and multiplicity function $\mathsf{k}: \mathcal{R}\to [0, \infty)$. Also, define the measure,
$$\mu_{\mathsf k}(dx)=\prod_{\alpha\in\mathcal{R}}|\langle x, \alpha\rangle|^{\mathsf{k}(\alpha)}\ dx= 
\prod_{\alpha\in\mathcal{R}_+}|\langle x, \alpha\rangle|^{2\mathsf{k}(\alpha)}\, dx :=w_{\mathsf k}(x)\, dx.$$
Let $\mathcal{G}$ be the Coxeter (reflection) group generated by $\mathcal{R}$. Note  that $\mathsf{k}$ is invariant under
$\mathcal{G}$, i.e. $\mathsf{k}(gv)=\mathsf{k}(v)$ for all  $g\in \mathcal{G}, v\in \mathcal{R}.$ The closure of the connected components of
$$
\{x\in\Rn\; :\; \langle x, \alpha\rangle\neq 0\quad \forall\; \alpha\in\mathcal{R}_+\}
$$
are called \textbf{\textit{Weyl chambers}}. We denote this Weyl
chambers by $\mathcal{W}_i$.
By $L^p(\Rn, \mu_{\mathsf k})$ we denote the Lebesgue space with respect to the measure $\mu_{\mathsf k}$. For $\xi\in\Rn$, the
\textit{\textbf{Dunkl operators}} $T_\xi$, as introduced by Dunkl in \cite{Dunkl}, are the following deformation of the directional derivative $\partial_\xi$:
$$
T_\xi f(x)=\partial_\xi f(x) +\sum_{\alpha\in\mathcal{R}_+} \mathsf{k}(\alpha)\langle \alpha, \xi\rangle\frac{f(x)-f(\sigma_\alpha(x))}{\langle \alpha, x\rangle}, $$
where $\sigma_\alpha$ denotes the reflection with respect to the hyperplane orthogonal to $\alpha$. The \textit{\textbf{Dunkl Laplacian} }
associated to $\mathcal{R}$ and $\mathsf{k}$ is given by $\Delta_{\mathsf k}=\sum_{j=1}^n T_j^2,$ where $T_j=T_{e_j}$ and 
$\{e_j\;, 1\leq j\leq n\}$ is the canonical orthonormal basis in $\Rn$. More precisely, for $f\in C^2(\Rn)$, we have
$$\Delta_{\mathsf k} f(x)=\Delta f(x) + \sum_{\alpha\in\mathcal{R}_+} 2\mathsf{k}(\alpha) \delta_\alpha f(x),$$
where 
$$\delta_\alpha f(x)= \frac{\partial_\alpha f(x)}{\langle \alpha, x\rangle}  -\frac{\norm{\alpha}^2}{2}
\frac{f(x)-f(\sigma_\alpha(x))}{\langle \alpha, x\rangle^2}.$$
It is known from \cite{Amri-Hammi, Rosler} that $\Delta_{\mathsf k}$ generates a contraction semigroup $\{D_t\}$ on $L^2(\Rn, \mu_{\mathsf k})$
and $D_t$ can also be extended to a bounded linear operator from $L^p(\Rn, \mu_{\mathsf k})$ to $L^\infty(\Rn, \mu_{\mathsf k})$
for $1\leq p\leq \infty$.

For the weight function \( w_{\mathsf k} \), there exists an isometry on the Hilbert space \( L^2(\Rn, \mu_{\mathsf k}) \), known as the Dunkl transformation (see \cite{Dunkl-transform}). This transformation exhibits properties analogous to those of the classical Fourier transform and is defined by the formula  
\[
\mathcal{F}_{\mathsf k}{f}(\xi) = c_\mathsf k \int_{\mathbb{R}^n} f(x) E_{\mathsf k}(-i \xi, x) w_{\mathsf k}(x) \, dx,
\]
where \( E_{\mathsf k}(x, y) \), referred to as the Dunkl kernel, is the unique solution of the following system of differential equations:  
\[
T_\xi f(x) = \langle y, \xi \rangle f(x), \quad f(0) = 1.
\]
Here, \( c_{\mathsf k} \) is the normalization constant given by  
\[
c_{\mathsf k}^{-1} = \int_{\mathbb{R}^n} e^{-\frac{1}{2}|x|^2} \mu_{\mathsf k}(dx).
\]
Since the measure \( \mu_{\mathsf k}(dx) \) is not invariant under standard translation, a generalized translation operator is introduced in the context of the Dunkl transform. This operator is defined on the Dunkl transform as  
\[
\mathcal{F}_{\mathsf k}({\tau}_y f)(\xi) = E_{\mathsf k}(y, -i \xi) \mathcal{F}_{\mathsf k}{f}(\xi), \quad \xi \in \mathbb{R}^n.
\]
In general, there is no explicit formula for \( \tau_y f \); however, an explicit expression is available when \( f \) is a radial function (see \cite[Theorem 6.3.2]{Dai-Xu}).
For \( f, g \in L^2(\mathbb{R}^n, \mu_{\mathsf k}) \), their Dunkl convolution can be defined in terms of the translation operator as follows  
\[
(f *_{\mathsf k} g)(x) = \int_{\mathbb{R}^n} f(y) \tau_x g(-y) \, \mu_\mathsf{k}(dy).
\]
It has been established in \cite{Thangavelu-Xu} that if \( g \) is a bounded radial function in \( L^1(\mathbb{R}^n, \mu_{\mathsf k}) \), the convolution \( f *_{\mathsf k} g \) extends to all \( f \in L^p(\mathbb{R}^n, \mu_{\mathsf k}) \), \( 1 \leq p \leq \infty \), as a bounded operator. Specifically, the following inequality holds,  
\[
\|f *_{\mathsf k} g\|_{L^p(\mathbb{R}^n, \mu_{\mathsf k})} \leq \|g\|_{L^1(\mathbb{R}^n, \mu_{\mathsf k})} \|f\|_{L^p(\mathbb{R}^n, \mu_{\mathsf k})}.
\]
Define \( d_{\mathsf k} = n + \sum_{\alpha \in \mathcal{R}_+} 2\mathsf{k}(\alpha) \).
The semigroup \( D_t := e^{t \Delta_{\mathsf{k}}} \), generated by the Dunkl Laplacian \( -\Delta_{\mathsf{k}} \), can be expressed in terms of a positive kernel \( q_t \). Specifically,  
\begin{equation}\label{Dunkl-kernel}
D_t f = f *_{\mathsf{k}} q_t,
\end{equation}
where \( q_t(x) \) is given by  
\[
q_t(x) = c_{\mathsf{k}}^{-1} (2t)^{-d_{\mathsf{k}}/2} e^{-|x|^2 / (4t)}.
\]
The Dunkl transform of the heat kernel satisfies  
$\mathcal{F}_\mathsf{k}{q}_t(\xi) = c_{\mathsf k}^{-1} e^{-t|\xi|^2}.$
Moreover, the action of the generalized translation operator on the heat kernel satisfies the relation  
\[
\tau_{\mathsf k}(-y) q_t(x) = c_{\mathsf k}^{-1} (2t)^{-d_{\mathsf k}/2} e^{-\frac{|x|^2 + |y|^2}{4t}} E_{\mathsf k}\left(\frac{x}{\sqrt{2t}}, \frac{y}{\sqrt{2t}}\right):=h_t(x, y),
\]
leading to
\begin{equation}\label{E2.12}
D_t f (x) = f *_{\mathsf{k}} q_t(x) = \int_{\Rn} f(y) h_t(x,y) \mu_{\mathsf k}(dy).
\end{equation}
Since 
$$\mu_{\mathsf k}(B(x, r))\simeq r^n\,  \Pi_{\alpha\in\mathcal{R}_+} (|\langle x, \alpha\rangle| + r)^{2\mathsf{k}(\alpha)},
\quad r>0,$$
$\mu_{\mathsf k}$ satisfies the volume doubling condition and one
could attempt to verify Assumption~\ref{Assump-1.2} in this case.
But we take advantage of the Dunkl transform and consider the natural
Hardy-Littlewood maximal function associated to the Dunkl convolution.
For a positive weight $v$, we define the weight class with respect
to $\mu_{\mathsf k}$.
For \( 1 \leq p < \infty \), the weighted Lebesgue space \( L^p_v(\mathbb{R}^n, \mu_{\mathsf k}) \) is defined as  
\[
L^p_v(\mathbb{R}^n, \mu_{\mathsf k}) = \left\{ f: \mathbb{R}^n \to \mathbb{R} : \|f\|_{L^p_v(\mathbb{R}^n, \mu_{\mathsf k})} = \|f v^{\frac{1}{p}} \|_{L^p(\mathbb{R}^n, \mu_{\mathsf k})} < \infty \right\}
\]
and we say that \( v \) belongs to the class \( D_p^{\mathsf k} \) if there exists \( t_0>0 \) such that  for each 
$\mathcal{W}_i$ there exists $x_i\in \mathcal{W}_i$ satisfying
\begin{equation} \label{E2.13}
\|v^{-\frac{1}{p}} h_{t_0}(x_i, \cdot)\|_{L^{p'}(\mathbb{R}^n, \mu_{\mathsf k})} < \infty.
\end{equation}
To present our results in the Dunkl setting, we define the local maximal operator as  
\[
D^*_R f(x) = \sup_{0 < t < R} \left| f *_{\mathsf{k}} q_t \right|, \quad \text{for } R > 0.
\]
We now adapt and verify Assumption \ref{Assump-1.2} within this setting.  
\begin{proposition}\label{Assump-Dunkl}
The following hold:
\begin{itemize}
\item[(i)] For every $f\in C_c(\Rn)$ we have $\lim_{t\to 0} f *_{\mathsf{k}} q_t(x)=f(x)$ for every $x\in\Rn$.
\item[(ii)] There exists $\eta\in (0, 1]$ such that for $R\in (0, T)$ and $x_1, x_2\in \mathcal{W}_i$, there exists a constant $C=C(R, x_1, x_2)$ satisfying
$$ h_{\eta R}(x_1, y)\leq C\, h_R (x_2, y)\quad \text{for all}\; y\in\Rn.$$
Moreover, for any compact set $K\subset\Rn$, there exists a positive constant $C=C(x, K, R)$ satisfying 
$\inf_{y\in K}h_R(x, y)\geq C$.
\item[(iii)] There exist $\gamma\geq 1$,
 a function $\Gamma:(0, \infty)\to (0, \infty)$ and a constant $C$, such that for any $R>0$ we have
$$D^*_R f(x)\leq C\left(\mathcal{M}^{\mathsf{k}}_{\Gamma(R)} f(x) +  f *_{\mathsf{k}} q_{\gamma R}(x)\right)\quad \text{for}\; x\in\Rn,$$
for every measurable $f\geq 0$, where $\mathcal{M}^{\mathsf{k}}_R$ denotes the local Dunkl maximal function defined by
\begin{equation}\label{Dunml-LMF}
\mathcal{M}^{\mathsf{k}}_R f(x)= \sup_{r\in (0, R)}\, 
\frac{1}{\mu_{\mathsf k}(B(x,r))} \left|f *_{\mathsf k} \chi_{B_r}(x)\right|.
\end{equation}
\end{itemize}
\end{proposition}

\begin{proof}
 (i) can be obtained following the proof of \cite[Theorem 6.4.2]{Dai-Xu}. More precisely, fix $x$ and $\varepsilon>0$. Choose
 a $g\in C^\infty_c(\Rn)$ satisfying 
 $$\sup_{\Rn}\{|f\ast_{\mathsf k} q_t-g\ast_{\mathsf k} q_t|,
 |f-g|\}< \frac{\varepsilon}{2},$$
 for all $t>0$ (use \cite[Theorem~6.4.1]{Dai-Xu}). Since
 $$g\ast_{\mathsf k} q_t(x) -g(x)=
 \int_{\Rn}(\tau_{\mathsf k}(y)g(x)-g(x))q_t(y)\mu_{\mathsf k}(dy),
 $$
 and $|\tau_{\mathsf k}(y)g(x)-g(x)|\lesssim |y|$ by \cite[Theorem~6.2.7]{Dai-Xu}, we get $g\ast_{\mathsf k} q_t(x)- g(x)\to 0$ as $t\to 0$. Hence 
 $$\limsup_{t\to 0}|f\ast_{\mathsf k} q_t(x)- f(x)|\leq \varepsilon. $$
(i) follows from the arbitrariness of $\varepsilon$.

To verify (ii), we first recall the distance function induced by $\mathcal{G}$. Let
$$ {\rm d}(x, y)=\min\{|x-g(y)| : \, g\in \mathcal{G}\}.$$
Also, let $\omega(x, r)=\mu_{\mathsf k}(B(x, r))$, that is, $\mu_{\mathsf k}$ volume of the ball $B(x, r)$. Now 
from \cite[Theorem~1.1]{DA23} there exist constants $0<c_u<c_l$ such that
\begin{equation}\label{E2.15}
\frac{1}{C}\, (\omega(x, \sqrt{t}))^{-1} e^{-c_l\frac{{\rm d}(x, y)^2}{t}} \Lambda(x, y, t)\leq
h_t(x, y) \leq C\, (\omega(x, \sqrt{t}))^{-1} e^{-c_u\frac{{\rm d}(x, y)^2}{t}} \Lambda(x, y, t)\quad t>0,
\end{equation}
where $C$ is a universal constant and $\Lambda$ is an appropriate positive function defined by \cite[eq (1.7)]{DA23}. We set
$\eta=\frac{c_u}{4c_l}\in (0, 1)$. For ${\rm d}(x_1, y)\geq {\rm d}(x_1, x_2)$, implying 
$2{\rm d}(x_1, y)\geq {\rm d}(x_2, y)$, we have
$$e^{-c_u\frac{{\rm d}(x_1, y)^2}{\eta R}}\leq e^{-c_l\frac{{\rm d}(x_2, y)^2}{R}},$$
and for ${\rm d}(x_1, y)< {\rm d}(x_1, x_2)$, implying ${\rm d}(x_2, y)< 2 {\rm d}(x_1, x_2)$
$$e^{-c_u\frac{{\rm d}(x_1, y)^2}{\eta R}}\leq 1\leq e^{c_l\frac{4{\rm d}(x_1,x_2)^2}{R}} e^{-c_l\frac{{\rm d}(x_2, y)^2}{R}}.$$
Again, since $\omega(x, r)\simeq r^n \prod_{\alpha\in\mathcal{R}_+}(|\langle x, \alpha\rangle| + r)^{2 \mathsf{k}(\alpha)}$, 
for some constant $\kappa_1=\kappa(x_1, x_2, R)$ we have
$$(\omega(x_1, \sqrt{\eta R}))^{-1} \leq \kappa_1  (\omega(x_2, \sqrt{R}))^{-1}.$$
Also, since for any reflection $\sigma_\alpha$ we have
$$ 1+ \frac{|x_2-\sigma_\alpha(y)|}{\sqrt{R}}\leq \left(1+\frac{|x_1-x_2|}{\sqrt{R}}\right)\left(1+ \frac{|x_1-\sigma_\alpha(y)|}{\sqrt{R}}\right),$$
from the definition of $\Lambda$ (see \cite[eq. (1.7) and (1.8)]{DA23}) it follows that
$$\Lambda(x_1, y, \eta R)\leq \Lambda(x_1, y, R)\leq \left(1+\frac{|x_1-x_2|}{\sqrt{R}}\right)^{4|\mathcal{G}|} \Lambda(x_2, y, R)$$
 Combining these estimate in \eqref{E2.15} we see that the first assertion in (ii) holds. The second one follows from the lower bound in \cite[Prop.~3.1]{DA23}.

For $R>0,$ $\alpha$ is chosen such that $ t \mapsto t^{-\frac{d_{\mathsf{k}}}{2}} e^{-\frac{r^2}{4 t}}$ is increasing on  $(0, R]$ for all $ r > \sqrt{\alpha R} $. For \( 0 < t < R \), we decompose \( q_t \) as \( q_t = q_t^1 + q_t^2 \), where  
$q_t^1 = q_t \chi_{\{|x| \leq (\alpha R)^{1/2}\}}.$
If \( j_0 \in \mathbb{Z} \) is such that \( 2^{j_0}t < R < 2^{j_0+1}t \), then we have  
\[
\begin{aligned}
q_t^1(x) &\leq q_t(x) \left(\chi_{\{|x| \leq (\alpha t)^{1/2}\}}(x) + \sum_{j=0}^{j_0} \chi_{\{(\alpha 2^j t)^{1/2} \leq |x| \leq (\alpha 2^{j+1} t)^{1/2}\}}(x)\right) \\
&\leq t^{-\frac{d_{\mathsf{k}}}{2}} \chi_{\{|x| \leq (\alpha t)^{1/2}\}}(x) + \sum_{j=0}^{j_0} \left(\alpha 2^j\right)^{\frac{d_{\mathsf{k}}}{2}} e^{-\frac{\alpha}{4} 2^j} \left(\alpha 2^j t\right)^{-\frac{n}{2}} \chi_{\{|x| \leq (\alpha 2^{j+1} t)^{1/2}\}}(x).
\end{aligned}
\]  
Thus, for \( f \geq 0 \), using the monotonicity of Dunkl convolution with respect to radial functions (see \cite[Theorem~6.3.2]{Dai-Xu}),
we obtain  
\[
\sup_{t < R} f *_{\mathsf{k}} q_t(x) \leq C \mathcal{M}^{\mathsf{k}}_{(\alpha R)^{1/2}} f(x),
\]  
where  
\[
C = \alpha^{\frac{d_{\mathsf{k}}}{2}} + \sum_{j=0}^\infty \left(\alpha 2^j\right)^{\frac{d_{\mathsf{k}}}{2}} e^{-\frac{\alpha}{4} 2^j} < \infty.
\]  
On the other hand, since \( q_t^2(x) \) is increasing in the time interval \( (0, R) \), we have  
\[
\sup_{0 < t < R} q_t^2 *_{\mathsf{k}} f(x) = q_R^2 *_{\mathsf{k}} f(x) \leq q_R *_{\mathsf{k}} f(x).
\]  
Combining these estimates, we obtain  
\[
D_R^* f(x) \leq C \left(\mathcal{M}^{\mathsf{k}}_{(\alpha R)^{1/2}} f(x) + q_R *_{\mathsf{k}} f(x)\right).
\]
Thus (iii) also holds.
\end{proof}

The main result in the Dunkl setting is as follows.

\begin{theorem}\label{T-Dunkl}
Let $v$ be a positive weight in $\mathbb{R}^n$,  $q_t$ be the heat kernel as mentioned above and $1 \leq p<\infty$. 
Then the  following statements are equivalent:
\begin{enumerate}
\setlength\itemsep{1em}
\item[\hypertarget{1}{(1)}]
 There exists \(R>0\) and a weight \(u\) such that the operator \(D^*_R\) maps \(L_v^p(\mathbb{R}^n, \mu_{\mathsf k})\) into \(L_u^p(\mathbb{R}^n, \mu_{\mathsf k})\) for \(p > 1,\) and maps \(L_v^1(\mathbb{R}^n, \mu_{\mathsf k})\) into weak \(L_u^1(\mathbb{R}^n, \mu_{\mathsf k})\) when \(p=1\).
\item[\hypertarget{2}{(2)}]
 There exists \(R>0\) such that \(f *_{\mathsf{k}} q_R(x)\) is finite for all $x$, and the limit \(\lim_{t \rightarrow 0^+} f *_{\mathsf{k}} q_t(x) = f(x)\) exists almost everywhere for all \(f \in L_v^p(\mathbb{R}^n, \mu_{\mathsf k})\).
\item[\hypertarget{3}{(3)}]
 There exists \(R>0\) such that \(f *_{\mathsf{k}} q_R(x)\) is finite for some $x_i\in\mathcal{W}_i$,  for each Weyl chamber $\mathcal{W}_i$.
\item[\hypertarget{4}{(4)}]
 The weight $v\in D_p^{\mathsf{k}}.$
\end{enumerate}
\end{theorem}
\begin{proof}
We will proceed analogous to the proof of Theorem \ref{T-1.3}. Specifically, we shall establish \( (4)\Rightarrow(1) \); the remaining implications will follow from Proposition \ref{Assump-Dunkl} and the proof of Theorem \ref{T-1.3}.

Since the boundedness of the operator \( q_R *_\mathsf{k} f \) follows by duality, it is enough, by Proposition \ref{Assump-Dunkl}(iii), to show that for \( R > 0 \) there exists a weight \( u \) so that the operator \( \mathcal{M}^{\mathsf{k}}_{R} \) maps \( L_v^p(\mathbb{R}^n, \mu_{\mathsf{k}}) \) into \( L_u^p(\mathbb{R}^n, \mu_{\mathsf{k}}) \) for \( p > 1 \), and maps \( L_v^1(\mathbb{R}^n, \mu_{\mathsf{k}}) \) into weak \( L_u^1(\mathbb{R}^n, \mu_{\mathsf{k}}) \) when \( p = 1 \). In view of \cite[Lemma 3.4~$(iv)\Rightarrow(i)$]{vivianiPAMS}, it suffices to show that for any annulus $E_k=\{2^{k-1}R\leq|x|\leq 2^{k} R\}$ and \( 0 < s < p < \infty, \)
\begin{equation}\label{Dunkl-vector}
\left\|\left(\sum_{j \in \mathbb{Z}}\left|\mathcal{M}^{\mathsf{k}}_{R} f_{j}\right|^{p}\right)^{\frac{1}{p}}\right\|_{_{L^s(E_k, \mu_{\mathsf{k}})}} \leq C_{p, s}\left(\sum_{j \in \mathbb{Z}}\left\|f_{j}\right\|_{_{L^p_v(\mathbb{R}^n, \mu_{\mathsf{k}})}}^{p}\right)^{\frac{1}{p}},
\end{equation}
where \( f=\left\{f_j\right\} \) be vector-valued function on \( \mathbb{R}^n \) and $C_{p, s}$ is a constant depending on $p$ and $s$. In view of \cite[Theorem 11]{Deleaval} (see also \cite{Amri-Sifi}), we notice that \( \mathcal{M}^{\mathsf{k}}_{R}\) is a  weak type $(1,1)$ operator in vector-valued setting.  
Let $f_j = f_j \chi_{B(0, 2^{k+1}R)} + f_j \chi_{B^c(0, 2^{k+1}R)}=f'_j + f^{\prime\prime}_j$. Using \cite[Theorem~6.3.2 and Theorem~2.3.4]{DA23} we see that
$$\mathcal{M}^{\mathsf{k}}_{R} f^{\prime\prime}_{j} =0\quad \text{in} \; E_k.$$
Therefore, applying Kolmogorov inequality \cite[Theorem 3.3.1,  p. 59]{Guzman},  we can obtain the estimate \eqref{Dunkl-vector}
(see \cite[Lemma 3.4~$(iv)\Rightarrow(i)$]{vivianiPAMS}). This proves (1).
\end{proof}

\subsection{Laplacian with Hardy potential $\mathcal{L}_b=-\Delta + b |x|^{-2}$}\label{hp}
In this section, we consider the heat equation with a Hardy potential
\begin{equation}\label{E-hardy}
\begin{split}
\partial_t u+ \mathcal{L}_bu=\partial_tu- \Delta u +\frac{b}{|x|^2} u & =0 \quad  \text{in}\; \Rn\times (0, \infty),
\\
u(x, 0) &= f \quad \text{in}\; \Rn,
\end{split}
\end{equation}
where  $b\in\mathbb{R}$.
We assume that
$$n\geq 3, \quad \text{and}\quad D:=b+\frac{(n-2)^2}{4}\geq 0.$$
As well-known \cite{Hardy}, for $1<p<\infty$, there exists a semigroup $\{S^H_t\}$ on $L^p$ whose generator is 
the closure of $(\mathcal{L}_b, C^\infty_c(\Rn\setminus\{0\})$. This semigroup also attains a heat kernel $p_t^H$ and the following
bounds hold \cite[Theorem~6.2]{Hardy}: 
\begin{equation}\label{H-bound}
\frac{1}{C} t^{-\frac{n}{2}} \left[\Bigl(\frac{|x|}{\sqrt{t}}\wedge 1\Bigr)\Bigl(\frac{|y|}{\sqrt{t}}\wedge 1\Bigr)\right]^\upgamma e^{-\frac{|x-y|^2}{\upbeta_2 t}}
\leq p^H_t(x, y)\leq C\, t^{-\frac{n}{2}} \left[\Bigl(\frac{|x|}{\sqrt{t}}\wedge 1\Bigr)\Bigl(\frac{|y|}{\sqrt{t}}\wedge 1\Bigr)\right]^\upgamma e^{-\frac{|x-y|^2}{\upbeta_1 t}}
\end{equation}
for $t>0, \; x, y\neq 0$, where $\upgamma= \sqrt{D} -\frac{n}{2} +1$, $0<\upbeta_2\leq\upbeta_1$ and $C$ is some universal constant. The solution
of \eqref{E-hardy} can be formally written as
$$u(x, t)=p^H_t\circledast f(x):=\int_{\Rn} p_t^H(x, y) f(y)\, dy.$$
Note that $u(x, t)$ is not defined for $x=0$. So we consider
 $\Omega:=\Rn\setminus\{0\}$.
As before, we characterize the class of weight functions $v$ so that  $p^H_t\circledast f(x)$ converges to $f(x)$, as $t\to 0$, almost surely for all $f\in L^p_v(\Omega)$. As earlier, we define the class of weight function as follows: for $1\leq p<\infty$, let
$$D^H_p=\{v>0\; :\; \exists \, t_0>0\; \text{such that}\; p^H_{t_0}(x,\cdot) v^{-\frac{1}{p}}\in L^{p'}(\Omega)\; \text{for some}\; x\in\Omega\}.$$
We also need the following local maximal operator associated to $\mathcal{L}_b=-\Delta+b|x|^{-2}:$
$$\mathsf{H}^*_Rf(x)=\sup_{0<t<R}|p^H_t\circledast f(x)|.$$
Now we state the main result of this section.
\begin{theorem}\label{T8.1}
Suppose $1\leq p<\infty$ and $v$ is a positive weight function in $\Omega$.
The following statements are equivalent.
\begin{itemize}
\setlength\itemsep{1em}
\item[\hypertarget{1}{(1)}]  There exists $R>0$ and a positive weight $u$ such that the operator $\mathsf{H}^*_R$
 maps $L_v^p(\Omega)$ into $L_u^p(\Omega)$ for $p > 1$, and maps $L_v^1(\Omega)$ into weak $L_u^1(\Omega)$ for $p=1$.
\item[\hypertarget{2}{(2)}] There exists $R>0$ such that $p^H_R\circledast f(x)$ is finite for all $x\in\Omega$, and the limit 
$\lim_{t \rightarrow 0^+} p^H_t\circledast f(x) = f(x)$  almost everywhere, for all $f \in L_v^p(\Omega)$.
\item[\hypertarget{3}{(3)}] There exists $R>0$ such that 
$p^H_R\circledast f(x)$ is finite for some $x$ and for all $f\in L^p_v(\Omega)$.
\item[\hypertarget{4}{(4)}] The weight $v\in D^H_p.$
\end{itemize}
\end{theorem}
To prove  Theorem~\ref{T8.1} we broadly follow the proof of Theorem~\ref{T-1.3}. Since the heat kernel estimate \eqref{H-bound} introduces 
additional singularity (compare with \eqref{E2.1}) and  $p^H_t\circledast 1\neq 1$, we need more delicate estimates to prove our result.
Let us start with the following lemma.
\begin{lemma}\label{L8.2}
Let $f\in C_c(\Omega)$. Then $\lim_{t\to 0} p^H_t\circledast f(x)=f(x)$ for all $x\in\Omega$.
\end{lemma}
\begin{proof}
Denote by $q(z)=\frac{b}{|z|^2}$ and recall that $\hat{p}_t(x)= (4\pi t)^{-\frac{n}{2}}e^{-\frac{|x|^2}{4t}}$. By Duhamel's principle 
we know that
\begin{equation}\label{Duhamel}
p^H_t(x, y)=\hat{p}_t(x-y) + \int_0^t \int_{\Rn} \hat{p}_{t-s} (x-z) q(z) p^H_s(z, y)\, dz\, ds.
\end{equation}
This can be seen as follows: For any $f\in C^\infty_c(\Omega)$, $u(x, t):=p^H_t\circledast f(x)$ is in $H^2_{\loc}(\Omega)$ \cite{Hardy}. Hence, by semigroup theory
$$ \partial_t u-\Delta u= -\frac{b}{|x|^2} u.$$
Since $|u(x, t)|\lesssim (|x|\wedge 1)^\upgamma$ and $\sup_{t\in (0, T)}|u(x, t)|$ decays exponentially as $|x|\to\infty$ 
using \eqref{H-bound}, combining with $\upgamma-2+n>0$, we see that $|x|^{-2}u\in L^{1}(\Rn\times(0, T))$. From 
\cite[Theorem~4.8]{Folland}
and uniqueness of distributional solution, it follows that 
$$u(x, t)=\hat{p}_t\ast f(x) + 
\int_0^t \hat{p}_{t-s}\ast [q(\cdot)u(\cdot, s)](x) dx,$$
where $\ast$ denotes the convolution. 
From the density of $C^\infty_c(\Omega)$ we then have \eqref{Duhamel}.
Using the symmetry $p^H_t(x, y)=p^H_t(y,x)$, \eqref{Duhamel} also gives
$$p^H_t(x, y)=\hat{p}_t(x-y) + \int_0^t \int_{\Rn} \hat{p}_{t-s} (y-z) q(z) p^H_s(x, z) dz ds.$$
Since 
$$ \lim_{t\to 0}\int_{\Rn} \hat{p}_t(x-y)f(y) dy=f(x)\quad \text{for all}\; x,$$
to complete the proof, it is enough to show that
$$I(t):= \int_{\Rn} \int_0^t \int_{\Rn} \hat{p}_{t-s} (y-z) |q(z)|\, p^H_s(x, z)  |f(y)| dz\, ds\, dy\to 0\quad \text{as}\quad t\to 0,$$
for $x\in\Omega$. Fix $x\in\Omega$. Since  $f \in C_c(\Omega)$, there exists $\delta_0 > 0$ such that  
\[
\operatorname{supp}(f) \subset B_{\delta_0}^c(0).
\]
We can decompose  $I$ into two parts as follows:  
\[
\begin{aligned}
I(t) &= \int_{|y| \geq \delta_0} \int_0^t \int_{|z| \geq \delta_x} \hat{p}_{t-s} (y-z) |q(z)|\, p^H_s(x, z)  |f(y)| \, dz \, ds \, dy \\
  &\quad + \int_{|y| \geq \delta_0} \int_0^t \int_{|z| < \delta_x} \hat{p}_{t-s} (y-z) |q(z)|\, p^H_s(x, z)  |f(y)| \, dz \, ds \, dy,
 \\
 &:= I_{1}(t) + I_{2}(t),
\end{aligned}
\]
where $\delta_{x}=\left(\frac{\delta_0}{2} \wedge \frac{|x|}{2}\right)$.

Consider  $I_1$ first. Since $|z| \geq \delta_x $, it follows that   $|q(z)|\leq \frac{|b|}{\delta_x^2} $. Moreover, the condition \( \delta_x \leq |z| \) implies 
$\frac{\delta_x}{\sqrt{t}} \leq \frac{|z|}{\sqrt{t}}$, giving us 
\[
\left(\quad \frac{\delta_x}{\sqrt{t}} \wedge 1\right) \leq \left(\frac{|z|}{\sqrt{t}} \wedge 1\right) \leq 1.
\]
Thus, we have  
\[
\left( \frac{|z|}{\sqrt{t}} \wedge 1 \right)^\upgamma \leq  
\begin{cases} 
1 & \text{for}\; \upgamma \geq 0, \\ 
\left( \frac{\delta_x}{\sqrt{t}} \wedge 1 \right)^\upgamma & \text{for}\; \upgamma < 0. 
\end{cases}
\]
For  $t < \delta_x^2$ we then  obtain
\[
\left( \frac{|x|}{\sqrt{t}} \wedge 1 \right)^\upgamma \leq 1 \quad \text{and} \quad \left( \frac{|z|}{\sqrt{t}} \wedge 1 \right)^\upgamma \leq 1.
\]
Consequently, using  \eqref{H-bound}, we estimate 
\[
\begin{aligned}
I_1(t) & = \int_{|y| \geq \delta_0} \int_0^t \int_{|z| \geq \delta_x} \hat{p}_{t-s} (y-z) |q(z)|\, p^H_s(x, z)  |f(y)|  \, dz \, ds \, dy 
\\
&\leq \frac{C |b|}{\delta^2_x} \int_{|y| \geq \delta_0} \int_0^t \int_{\mathbb{R}^n} s^{-\frac{n}{2}} e^{-\frac{|x-z|^2}{\upbeta_1 s}} (t-s)^{-\frac{n}{2}} e^{-\frac{|z-y|^2}{4(t-s)}} |f(y)| \, dz \, ds \, dy 
\\
&\leq  \frac{C |b|}{\delta^2_x} \int_{|y| \geq \delta_0} \int_0^t \int_{\Rn} s^{-\frac{n}{2}} e^{-\frac{|x-z|^2}{c s}} (t-s)^{-\frac{n}{2}} e^{-\frac{|z-y|^2}{c(t-s)}} |f(y)|\, dz \, ds \, dy,
\end{aligned}
\]
where $c = \upbeta_1 \vee  4$. Using the Markov property  of heat kernel
\[
\int_{\mathbb{R}^n} \hat{p}_s(x-z) \hat{p}_{t-s}(z-y) \, dz = \hat{p}_{t}(x-y),
\]
we deduce that 
\[
I_1(t) \leq \frac{C |b|}{\delta^2_x} \int_{|y| \geq \delta_0} \int_0^t \hat{p}_{\frac{c}{4}t}(x-y) |f(y)| \, ds \, dy 
\leq t   \frac{C |b|}{\delta^2_x} \|f\|_\infty \to 0 \quad \text{as } t \to 0.
\]

Next we consider  $I_2$.  Letting $t < |x|^2$, it follows that  
\[
\left(\frac{|x|}{\sqrt{t}} \wedge 1\right)^\upgamma \leq 1.
\]
Since $|z|<\delta_x$ and $|y|\geq \delta_0$, it follows that $ \frac{\delta_0}{2} < |y| - |z| \leq |y - z| $ and \( \frac{|x|}{2} \leq |x| - |z| \leq |x - z| \), 
giving us 
\begin{equation}\label{E8.3}
p_s^H(x, z) |q(z)| \hat{p}_{t-s}(z-y) 
\leq C(\delta_0) \left(\frac{|z|}{\sqrt{s}} \wedge 1\right)^\upgamma \frac{|b|}{|z|^2} s^{-\frac{n}{2}} e^{-\frac{|x|^2}{ 4\upbeta_1 s}},
\end{equation}
where we used the fact that $\sup_{t>0} t^{-\frac{n}{2}} e^{-\frac{\delta^2_0}{4 t}}<\infty$.
For $\upgamma \geq 0$, we see that  
\[
p_s^H(x, z) |q(z)| \hat{p}_{t-s}(z-y) \leq C(x, \delta_0) \frac{1}{|z|^2},
\]
where the constant $C(x, \delta_0)$ depends on $x, \delta_0$. Hence,  
\[
I_2(t) \leq C(x, \delta_0)\, \norm{f}_\infty\, |{\rm supp}(f)|\, t \int_{|z| < \delta_x} \frac{1}{|z|^2} \, dz   \to 0 \quad \text{as } t \to 0.
\]
Now suppose $\upgamma<0$. Since
$$
\left(\frac{|z|}{\sqrt{s}} \wedge 1\right)^\upgamma
\leq \left\{\begin{array}{ll}
\frac{|z|^\gamma}{s^{\gamma/2}} & \text{when}\; |z| \leq \sqrt{s},
\\
1 & \text{when}\; |z| \geq \sqrt{s},
\end{array}
\right.
$$
it follows that (since $\upgamma<0$)
\[
\left(\frac{|z|}{\sqrt{s}} \wedge 1\right)^\upgamma \leq \frac{|z|^\upgamma}{s^{\upgamma/2}} + 1\leq t^{-\upgamma/2}|z|^\upgamma + 1.
\]
Since  $\upgamma - 2 + n= n-2 +\sqrt{D}-\frac{n}{2} +1\geq \frac{n}{2} -1> 0$ , $|z|^{\upgamma-2}$ becomes integrable at $0$.
Thus using \eqref{E8.3} and the fact $\sup_{s>0} s^{-\frac{n}{2}}e^{-\frac{|x|^2}{4\upbeta_1 s}}<\infty$,
we obtain  
$I_2(t) \to 0$ as $t\to 0$. Combining the estimates of $I_1$ and $I_2$ we complete the proof.
\end{proof}

For our next result, we introduce a local Hardy-Littlewood functional with a weight function.
\begin{equation}\label{H-LMF}
\mathcal{M}^H_R f(x):=\sup_{s\in (0, R)}\, \frac{1}{|B(x,r)|}\int_{B(x,s)} (\abs{y}\wedge 1)^\upgamma f(y)\, dy.
\end{equation}
The next lemma is a counterpart of Assumption~\ref{Assump-1.2}(iii) in the Hardy operator setting.
\begin{lemma}\label{L8.3}
There exist  locally bounded functions $\xi_1, \xi_2:\Omega\times(0, \infty)\to (1, \infty)$, 
 $\xi_i(\cdot, r), i=1,2$ are locally bounded in $\Omega$ for each $r>0$, and a  function $\Gamma:(0, \infty)\to(0, \infty)$
such that
for any $R>0$ we have
$$\mathsf{H}^*_R f(x)\leq \xi_1(x, R) \mathcal{M}^H_{\Gamma(R)} f(x) + \xi_2(x, R) \int_{\Rn} P^H_{\gamma R}(x, y) f(y) dy\quad \text{for}\; x\in\Omega,$$
for every measurable $f\geq 0$, where $\gamma=\frac{\upbeta_1}{\upbeta_2}$ .
\end{lemma}

\begin{proof}
We split the proof in two cases.

\noindent{\bf  Case 1.} Let $\upgamma\leq 0$. Let us denote  
\[
\chi_R(x, y) = \chi_{\{|x - y| \leq \sqrt{\alpha R}\}}(x, y),
\]  
where $\alpha$ is chosen such that $ t \mapsto t^{-\frac{n}{2}} e^{-\frac{r^2}{\upbeta_1 t}}$ is increasing on  $(0, R]$
 for all $ r > \sqrt{\alpha R} $. In particular, we may choose any $\alpha\geq \frac{n\upbeta_1}{2}$. We also denote by
 $$\chi^c_R(x, y)=\chi_{\{|x - y| > \sqrt{\alpha R}\}}(x, y).$$
 Due to our choice of $\alpha$, we have for $\gamma=\frac{\upbeta_1}{\upbeta_2}$ that
 \begin{align*}
 \sup_{t\leq R} p^H_t(x, y)\chi^c_R(x, y) & \leq C\,  \left[\left(\frac{|x|}{\sqrt{R}}\wedge 1\right)\left(\frac{|y|}{\sqrt{R}}\wedge 1\right)\right]^\upgamma
 R^{-\frac{n}{2}} e^{-\frac{|x-y|^2}{\upbeta_1 R}}
 \\
 &\leq C \gamma^{\frac{n}{2}} 
 \left[\left(\frac{|x|}{\sqrt{\gamma R}}\wedge 1\right)\left(\frac{|y|}{\sqrt{\gamma R}}\wedge 1\right)\right]^\upgamma
 (\gamma R)^{-\frac{n}{2}} e^{-\frac{|x-y|^2}{\upbeta_2 \gamma R}}
 \\
 &\leq C^2\gamma^{\frac{n}{2}} p^H_{\gamma R}(x, y).
 \end{align*}
 Hence, for all $f\geq 0$, we get
 \begin{equation}\label{E8.4}
 \sup_{0 < t < R} \int_{\mathbb{R}^n} p^H_t(x, y)\chi^c_R(x,y) f(y) \, dy 
 \leq C^2\gamma^{\frac{n}{2}} \int_{\mathbb{R}^n} p^H_{\gamma R}(x, y) f(y) \, dy.
\end{equation}
Since
$$\frac{|y|}{\sqrt{R}}\wedge 1\geq \frac{1}{\sqrt{R\vee 1}}\left(|y|\wedge 1\right),$$
for  $t < R $, we have  
\begin{align}\label{E8.5}
p_t^H(x, y) & \leq C t^{-\frac{n}{2}} \left(\frac{|x|}{\sqrt{t}} \wedge 1\right)^\gamma \left(\frac{|y|}{\sqrt{t}} \wedge 1\right)^\upgamma e^{-\frac{|x-y|^2}{\upbeta_1 t}}\nonumber
 \\
& \leq C t^{-\frac{n}{2}} \left(\frac{|x|}{\sqrt{R}} \wedge 1\right)^\upgamma \left(\frac{|y|}{\sqrt{R}} \wedge 1\right)^\upgamma e^{-\frac{|x-y|^2}{\upbeta_1 t}}
\nonumber
 \\
& \leq C (R\vee 1)^{-\frac{\upgamma}{2}}\left(\frac{|x|}{\sqrt{R}} \wedge 1\right)^\upgamma (|y| \wedge 1)^\upgamma t^{-\frac{n}{2}} e^{-\frac{|x-y|^2}{\upbeta_1 t}}.
\end{align}
Now we can proceed as Lemma~\ref{L2.3}. For $t\in (0, R)$, choose $j\in\mathbb{N}\cup\{0\}$ so that $2^{j} t< R\leq 2^{j+1} t$. 
Define $r_j=\Gamma(t2^j),$ where $\Gamma(r):=\sqrt{\alpha r}$. Then, using \eqref{E8.5}
\begin{align}
&p_t^H(x, y)\chi_R(x, y) \nonumber
\\
&\leq p^H_t(x, y) \chi_{\{|x-y|\leq r_{0}\}} + \sum_{k\leq j+1} p^H_t(x, y) \chi_{\{r_{k-1}<|x-y|\leq r_{k}\}}\nonumber
\\
&\leq C (R\vee 1)^{-\frac{\upgamma}{2}}\left(\frac{|x|}{\sqrt{R}} \wedge 1\right)^\upgamma\Bigl[
\alpha^{\frac{n}{2}} \frac{1}{r^n_0} (|y| \wedge 1)^\upgamma \chi_{\{|x-y|\leq r_{0}\}}
+ \sum_{k=1}^{j+1} (\alpha 2^k)^{\frac{n}{2}} e^{-\frac{\alpha 2^k}{2\upbeta_1}} (|y| \wedge 1)^\upgamma \frac{1}{r^n_k}\chi_{\{|x-y|\leq r_{k}\}}
\Bigr].
\end{align}
Thus, for $f\geq 0$,
$$\sup_{0<t<R} \int_{\Rn} p^H_t(x, y) \chi_R(x, y) f(y)\, dy
\leq \xi_1(x, R) \mathcal{M}^H_{\Gamma(R)} f(x),$$
where
$$\xi_1(x, R) = C (R\vee 1)^{-\frac{\upgamma}{2}}\left(\frac{|x|}{\sqrt{R}} \wedge 1\right)^\upgamma \left(\alpha^{\frac{n}{2}}
+ \sum_{k\geq 1} (\alpha 2^k)^{\frac{n}{2}} e^{-\frac{\alpha 2^k}{2\upbeta_1}}\right).$$
Using the  above estimate together with \eqref{E8.4}, we obtain the conclusion of the lemma.

\noindent{\bf  Case 2.} Let $\upgamma> 0$. Then for $t\leq R$, we get from \eqref{H-bound} that
\[
p_t^H(x, y) \chi_R^c(x, y) \leq C \left(\frac{|y|}{\sqrt{t}} \wedge 1\right)^\upgamma t^{-\frac{n}{2}} e^{-\frac{|x-y|^2}{\upbeta_1 t}} \chi_R^c(x, y).
\]
We set $\alpha$ large enough so that 
$$t\mapsto t^{-\frac{n+\upgamma}{2}} e^{-\frac{r^2}{\upbeta_1 t}}\quad \text{and}
\quad t\mapsto t^{-\frac{n}{2}} e^{-\frac{r^2}{\upbeta_1 t}}$$
are increasing in $(0, R]$
for all $r\geq\sqrt{\alpha R}$. Also, let $\gamma=\upbeta_1/\upbeta_2$.
 Let $y\in \{|x-y|>\sqrt{\alpha R}\}\cap\{|y|\leq \sqrt{\gamma R}\}$. Then
\begin{align*}
p_t^H(x, y) \chi_R^c(x, y) & \leq C |y|^\upgamma t^{-\frac{\upgamma}{2} - \frac{n}{2}} e^{-\frac{|x-y|^2}{\upbeta_1 t}} 
\\
& \leq C |y|^\upgamma R^{-\frac{\upgamma}{2} - \frac{n}{2}} e^{-\frac{|x-y|^2}{\upbeta_1 R}} 
\\
& \leq C \gamma^{\frac{\upgamma}{2}}\left(\frac{|y|}{\sqrt{\gamma R}} \wedge 1\right)^\upgamma R^{-\frac{n}{2}} e^{-\frac{|x-y|^2}{\upbeta_1 R}} 
\\
& \leq C^2 \gamma^{\frac{n+\upgamma}{2}} \left(\frac{|x|}{\sqrt {\gamma R}}\wedge 1\right)^{-\upgamma} p_{\gamma R}^H(x, y).
\end{align*}
On the other hand, for $y\in \{|x-y|>\sqrt{\alpha R}\}\cap\{|y|> \sqrt{\gamma R}\}$, we estimate
\begin{align*}
p_t^H(x, y) \chi_R^c(x, y) & \leq C R^{-\frac{n}{2}} e^{-\frac{|x-y|^2}{\upbeta_1 R}} 
\\
& \leq C \left(\frac{|y|}{\sqrt{\gamma R}} \wedge 1\right)^\upgamma R^{-\frac{n}{2}} e^{-\frac{|x-y|^2}{\upbeta_1 R}}
\\
&\leq C^2\gamma^{\frac{n}{2}} \left(\frac{|x|}{\sqrt{\gamma R}}\wedge 1\right)^{-\upgamma} p_{\gamma R}^H(x, y).
\end{align*}
Letting $\xi_2(x, R)= C^2 \gamma^{\frac{n+\upgamma}{2}} \left(\frac{|x|}{\sqrt{\gamma R}}\wedge 1\right)^{-\upgamma}$ we conclude that 
\begin{equation}\label{E8.7}
\sup_{0 < t < R} \int_{\mathbb{R}^n} p_t^H(x, y) \chi_R^c(x, y) f(y) \, dy
 \leq \xi_2(x, R) \int_{\mathbb{R}^n} p_{\gamma R}^H(x, y) f(y) \, dy
\end{equation}
for all $f\geq 0$. Again, by \eqref{H-bound}
\[
p_t^H(x, y) \chi_R(x, y) \leq C \left(\frac{|y|}{\sqrt{t}} \wedge 1\right)^\gamma t^{-\frac{n}{2}} e^{-\frac{|x-y|^2}{\upbeta_1 t}} \chi_R(x, y).
\]
First consider $y\in \{|x-y|\leq\sqrt{\alpha R}\}\cap\{|y|\geq \frac{|x|}{2}\wedge 1\}$.
Then  
\[
\left(\frac{|y|}{\sqrt{t}} \wedge 1\right) \leq \left(\frac{|x|}{2}\wedge 1\right)^{-1} (|y| \wedge 1)
\]
as the right-hand side is larger than $1$.
Thus,  for the above choice of $y$, 
\[
p_t^H(x, y) \chi_R(x, y) \leq \left(\frac{|x|}{2}\wedge 1\right)^{-\upgamma} (|y| \wedge 1)^\upgamma t^{-\frac{n}{2}} e^{-\frac{|x-y|^2}{\upbeta_1 t}} \chi_R(x, y).
\]
Now consider $y\in \{|x-y|\leq\sqrt{\alpha R}\}\cap\{|y|\leq \frac{|x|}{2}\wedge 1\}$. 
Then $|x-y| \geq \frac{|x|}{2} $. Using this, we get  
\[
p_t^H(x, y) \chi_R(x, y) \leq (|y|\wedge 1)^\upgamma t^{-\frac{\upgamma}{2}{-\frac{n}{2}}} e^{-\frac{|x|^2}{\upbeta_1 t}} \chi_R(x, y).
\]
Defining  $ \tilde\xi(x, R) = (\alpha R)^{n/2} \sup_{t\leq R} t^{-\frac{\gamma}{2}{-\frac{n}{2}}} e^{-\frac{|x|^2}{\upbeta_1 t}}$, we find  
\[
p_t^H(x, y) \chi_R(x, y) \leq \tilde\xi(x, R) \frac{1}{(\alpha R)^{n/2}} (|y| \wedge 1)^\gamma \chi_R(x, y).
\]
Combining the above two cases, we have
$$p_t^H(x, y) \chi_R(x, y) \leq \left(\frac{|x|}{2}\wedge 1\right)^{-\upgamma} (|y| \wedge 1)^\upgamma t^{-\frac{n}{2}} e^{-\frac{|x-y|^2}{\upbeta_1 t}} \chi_R(x, y)
+
\tilde\xi(x, R) \frac{1}{(\alpha R)^{n/2}} (|y| \wedge 1)^\gamma \chi_R(x, y).
 $$
 At this state, we define 
 $$\xi_1(x, R)= \left(\frac{|x|}{2}\wedge 1\right)^{-\upgamma}  \left(\alpha^{\frac{n}{2}}
+ \sum_{k\geq 1} (\alpha 2^k)^{\frac{n}{2}} e^{-\frac{\alpha 2^k}{2\upbeta_1}}\right) + \tilde\xi(x, R)$$
and repeating the calculation of Case 1, we arrive at 
\[
\sup_{0 < t < R}  \int_{\mathbb{R}^n} p_t^H(x, y) \chi_R(x, y) f(y) \, dy  \leq \xi_1(x, R) \mathcal{M}^H_{\Gamma(R)} f(x),
\]
for all $f\geq 0$. The above estimate together with \eqref{E8.7} complete the proof.
\end{proof}

We also need the following lemma to prove Theorem~\ref{T8.1} (compare it with \cite[Lemma~3.4]{vivianiPAMS}).
\begin{lemma}\label{L8.4}
Let $1\leq  p<\infty$ and $v\in D^H_p$. For every $R>0$,  there exists a positive weight $u$ so that $\mathcal{M}^H_R$
maps $L^p_v(\Omega)$ into $L^p_u(\Omega)$ for $p>1$ and $L^1_v(\Omega)$ into weak $L^1_u(\Omega)$ for $p=1$.
\end{lemma}

\begin{proof}
Since $v\in D^H_p$, from \eqref{H-bound} it follows that $(|y|\wedge 1)^\upgamma v^{-\frac{1}{p}}\in L^{p'}_{\loc}(\Rn)$. Let
$\tilde{v}(y) =v(y) (|y|\wedge 1)^{-p\upgamma}$. Then $\tilde{v}^{-\frac{1}{p}}\in L^{p'}_{\loc}(\Rn)$. By Proposition \ref{blmx},
there exists a weight $u$ such that $\mathcal{M}_R$ maps $L^p_{\tilde v}$ into $L^p_{u}$ for $p>1$ and 
$L^1_{\tilde v}$ into weak $L^1_{u}$ for $p=1$. Let $f\in L^p_{v}(\Omega)$. Define $\tilde{f}(y)=f(y) (|y|\wedge 1)^\upgamma$.
Then
$$\int_{\Rn} |\tilde{f}(y)|^p \tilde{v}(y)\, dy= \int_{\Rn} |f(y)|^p v(y)\, dy<\infty.$$
Hence $\tilde{f}\in L^p_{\tilde{v}}$. Let $p>1$. Then, for some constant $\kappa$, we have
$$ \int_{\Rn} |\mathcal{M}^H_R f(x)|^p u(x) \, dx=
\int_{\Rn} |\mathcal{M}_R \tilde{f}(x)|^p u(x) \, dx\leq \kappa\; \int_{\Rn} |\tilde{f}(x)|^p \tilde{v}(x) \, dx
=\int_{\Rn} |f(y)|^p v(y)\, dy.$$
Thus $\mathcal{M}^H_R$ maps $L^p_v(\Omega)$ into $L^p_u(\Omega)$ for $p>1$. Similarly, we can establish the conclusion for
$p=1$. This completes the proof.
\end{proof}

Now we can complete the proof of Theorem~\ref{T8.1}.
\begin{proof}[\textbf{Proof of Theorem~\ref{T8.1}}]
Note that, for the operator \eqref{E-hardy}, Assumption~\ref{Assump-1.2} holds. More precisely, Assumption~\ref{Assump-1.2}(i) holds by
Lemma~\ref{L8.2}, \eqref{H-bound} confirms Assumption~\ref{Assump-1.2}(ii) and Assumption~\ref{Assump-1.2}(iii) holds due to
Lemma~\ref{L8.3}. Now we can complete the proof of Theorem~\ref{T8.1} along the lines of Theorem~\ref{T-1.3} with the help of
Lemma~\ref{L8.4}.
\end{proof}


\section{Abstract nonhomogeneous heat equation}\label{S-nonhom}
In this section, we again consider the operator in \eqref{AB1} but with a nonhomogeneous  term. More precisely, we consider
\begin{equation}\label{CD1}
\begin{split}
\partial_t u + \mathcal{L} u & = F(x, t) \quad \text{in}\; \cX\times \mathbb{R}_+
\\
u(x, 0) &= 0.
\end{split}
\end{equation}
We assume the setting of Section~\ref{pre}. Formally, by Duhamel's principle, the solution
of \eqref{CD1} is written as
\begin{equation}\label{E9.2}
u(x, t) = \varphi_t\odot F(x, t):= \int_0^t \int_{\cX}\varphi_{t-s}(x, y) F(y, s)\, d\mu(y)\, ds.
\end{equation}
As before, our goal is to characterize the class of weight functions for which $\varphi_t\odot F(x, t)\to 0$ almost everywhere as $t\to 0$. 
We consider weights of the form $wv,$ where $v:\cX\to (0, \infty)$ and $w:\mathbb{R}_+\to (0, \infty)$. For $1 \leq q<\infty$ and $1<p<\infty$, we
define the space
$$L^q_w((0, T), L^p_v(\cX))=\left\{F:\cX\times\mathbb{R}_+\to\mathbb{R}\;
:\; \int_0^T \norm{F(\cdot, s)}^q_{L^p_v(\cX)} w(s)\, ds<\infty \right\}.$$
The weight class $D_{q, p}$ in this case if defined as follows: 
\begin{definition}\label{wnp}
A weight $wv\in D_{q, p}$ if and only if there 
exists a time $t_0\in (0, T)$ such that
$$\int_0^{t_0}  \norm{\varphi_{t-s}(x, \cdot) v^{-\frac{1}{p}} (\cdot)}^{q'}_{L^{p'}(\cX)} w^{-\frac{q'}{q}}(s)\, ds<\infty
\quad \text{for almost every}\; x\in\cX.$$
\end{definition}
Here $\frac{1}{p}+\frac{1}{p'}=1$ and $\frac{1}{q}+\frac{1}{q'}=1$. The local maximal operator associated to \eqref{E9.2}
is defined as  
\[
\mathsf{NH}^*_R F(x) = \sup_{0 < t < R}| \varphi_t\odot F(x, t)|
= \sup_{0 < t < R} \left|\int_{0}^{t} \int_{\cX} \varphi_{t-s}(x, y) F(y, s) \, d\mu(y) \, ds \right|,\quad R>0.
\]  
We also modify Assumption~\ref{Assump-1.2} as follows.
\begin{assumption}\label{Assump-9.1}
The following hold:
\begin{itemize}
\item[(i)] For every $f\in C_c(\cX)$ we have
$$\lim_{t\to 0^+} \int_{\cX} \varphi_t (x, y) f(y)\,d\mu(y)=f(x)$$
for every $x\in\cX$.
\item[(ii)] For every compact sets $K_1\subset (0, T)$ and $B(x_0, R)$, where  $R>0$, there exists
$C=C(x, R, K_1)$ satisfying $\varphi_{t} (x, y)\geq C$ for all $t\in K_1, y\in B(x_0, R)$.

\item[(iii)] There exist $\gamma\geq1$, a function $\Gamma:(0, \infty)\to (0, \infty)$ and functions $\xi_1, \xi_2: \cX\times (0, \infty)\to [1, \infty)$,
with $\xi_i(\cdot, R), i=1,2,$ are locally bounded for each $R$, such that for any $R\in (0, T/\gamma)$ we have
$$\mathsf{NH}^*_R F(x)\leq \xi_1(x, R)\, \mathcal{M}_{R, \Gamma(R)} F(x) + \xi_2(x, R) \,\varphi_{\gamma R}\odot F(x, \gamma R) \quad \text{for}\; x\in\cX,$$
for every measurable $F\geq 0$, where $\mathcal{M}_{R_1, R_2}$ denotes the local Hardy-Littlewood maximal function defined by
$$ 
\mathcal{M}_{R_1, R_2} F(x):=\int_0^{R_1}\mathcal{M}^{\cX}_{R_2}[F(\cdot,t)](x)\,dt= \int_0^{R_1} \left[\sup_{s\in (0, R_2)}\, \frac{1}{\mu(B(x, s))} \int_{B(x, s)} |F(y, t)|\, d\mu(y)\right] dt.
$$
\end{itemize}
\end{assumption}

The main result of this section is as follows. 
\begin{theorem}\label{T-9.1}
Let  $\{\varphi_t\}$ be the heat kernel as mentioned above, $wv$ be a positive weight and $1 \leq q<\infty$, $1<p<\infty$. 
Suppose that Assumption~\ref{Assump-9.1} holds.
Then the  following statements are equivalent:
\begin{enumerate}
\setlength\itemsep{1em}
\item[\hypertarget{11}{(1)}]
 There exists $R>0$ and a weight $u$ such that the operator 
 $\mathsf{NH}^*_R$ maps \(L^q_w((0, R), L_v^p(\cX))\) into \(L_u^p(\cX)\) for \(p > 1\).
\item[\hypertarget{12}{(2)}]
 There exists \(R>0\) such that \(\varphi_R\odot F(x, R)\) is finite for almost every $x$, and the 
 limit \(\lim_{t \rightarrow 0^+} \varphi_t\odot F(x, t) = 0\)  almost everywhere, for all 
 \(F \in L^q_w((0, R), L_v^p(\cX))\).
\item[\hypertarget{13}{(3)}]
 There exists \(R>0\) such that \(\varphi_R\odot F(x, R)\) is finite for almost every $x$ and for all 
 \(F \in L^q_w((0, R), L_v^p(\cX))\).
\item[\hypertarget{14}{(4)}]
 The weight $w v\in D_{q,p}.$
\end{enumerate}
\end{theorem}

Before we proceed to prove Theorem~\ref{T-9.1}, let us verify that the operators in Example~\ref{Eg-2.1}, ~\ref{Eg2.2} and ~\ref{Eg9.6} satisfy Assumption~\ref{Assump-9.1}.
It is easy to see that Assumption~\ref{Assump-9.1}(i) and (ii) are met. So we discuss Assumption~\ref{Assump-9.1}(iii). Consider Example~\ref{Eg-2.1} first.
Set $\rho=\frac{\alpha_1}{\alpha_2}$ and $\gamma=2\rho$.
For some $\alpha>0$, to be chosen later, we split
$$\varphi_{t-s}(x, y)= \varphi_{t-s}(x, y) \chi_{\{|x-y|\leq \sqrt{\alpha R}\}} + \varphi_{t-s}(x, y)  \chi_{\{|x-y|> \sqrt{\alpha R}\}}.$$
Choose $\alpha$ large enough so that 
$$t\mapsto t^{-\frac{n}{2}} e^{-\frac{r^2}{t\alpha}}$$
is increasing in $[0, \gamma R]$ for $r\geq \sqrt{\alpha R}$. Then any $s\leq t\leq R$ we have (see \eqref{E2.1})
$$\varphi_{t-s}\chi_{\{|x-y|> \sqrt{\alpha R}\}}\leq C (2R-s\rho^{-1})^{-\frac{n}{2}} e^{-\frac{|x-y|^2}{\alpha_1 (2R-s\rho^{-1})}}\chi_{\{|x-y|> \sqrt{\alpha R}\}}
\leq C^2 \rho^{n/2}  \varphi_{\gamma R-s}\chi_{\{|x-y|> \sqrt{\alpha R}\}} .$$
Thus, 
$$\sup_{t\leq R} \int_0^t \int_{\cX}\varphi_{t-s}(x, y)  \chi_{\{|x-y|> \sqrt{\alpha R}\}} F(y, s)\, d\mu(y)\, ds
\leq C^2 \rho^{n/2} \varphi_{\gamma R}\odot F (x, \gamma R).$$
The estimate on the first term can be computed as before
(i.e., splitting the domain in annuli and dominating the kernel by maximal functions). Similar computation also applies for Example~\ref{Eg2.2} and ~\ref{Eg9.6}.

To prove Theorem~\ref{T-9.1} we need the following lemma which is an extension of Proposition~\ref{blmx}.

\begin{lemma}\label{L9.3}
Let $wv\in D_{q, p}$ for some $1\leq q<\infty$ and $1<p<\infty$. Then, there exists $R_\circ\in (0, T)$ such that for any $R>0$, there exists a positive weight $u$ so that $\mathcal{M}_{R_\circ, R}$ maps
$L^q_w((0, R_\circ), L^p_v(\cX))$ into $L^p_u(\cX)$ for $p>1$.
\end{lemma}

\begin{proof}
By definition of the weight class, there exists $t_0\in (0, T)$ such that
$$\int_0^{t_0}  \norm{\varphi_{t-s}(x, \cdot) v^{-\frac{1}{p}} (\cdot)}^{q'}_{L^{p'}} w^{-\frac{q'}{q}}(s)\, ds<\infty
\quad \text{for almost every}\; x\in\cX.$$
Fix $R_\circ<t_0$. Due to Assumption~\ref{Assump-9.1}(ii) we have $v^{-\frac{1}{p}}\in L^{p'}_{\loc}(\cX)$ and 
$$\int_0^{R_\circ} (w(s))^{-\frac{q'}{q}} \, ds<\infty.$$
Let $p>1$.
Applying Proposition~\ref{blmx} we obtain a weight $u$ such that
$$\left[\int_{\cX} |\mathcal{M}^{\cX}_R f(x)|^p u(x) \, d\mu(x)\right]^{\frac{1}{p}}\leq C \norm{f}_{L^p_v(\cX)}\quad \text{for all}\; f\in L^p_v(\cX).$$
Consider $F\in L^q_w((0, R_\circ), L^p_v(\cX))$. Then $F(\cdot, s)\in L^p_v(\cX)$ for almost every $s$, giving us,
\begin{align*}
\left[\int_{\cX} |\mathcal{M}_{R_\circ,R} F(x)|^p u(x) \, d\mu(x)\right]^{\frac{1}{p}} &\leq \left[\int_{\cX} (\mathcal{M}_{R_\circ,R} |F|(x))^p u(x) \, d\mu(x)\right]^{\frac{1}{p}}
\\
&=\left[\int_{\cX} \left(\int_0^{R_\circ} \mathcal{M}^{\cX}_{R} [|F|(\cdot,s)](x) \, ds\right)^p u(x) \, d\mu(x)\right]^{\frac{1}{p}}
\\
& \leq \int_0^{R_\circ} \left[\int_{\cX}  \left(\mathcal{M}^{\cX}_{R} [|F|(\cdot,s)](x) \right)^p u(x) \, d\mu(x)\right]^{\frac{1}{p}}\, ds
\\
&\leq \int_0^{R_\circ} C \norm{F(\cdot, s)}_{L^p_v(\cX)} \, ds
\\
&\leq C \left[\int_0^{R_\circ} w^{-\frac{q'}{q}}(s)\, ds\right]^{\frac{1}{q'}}
\left[\int_0^{R_\circ} C \norm{F(\cdot, s)}^q_{L^p_v(\cX)} w(s) \, ds\right]^{1/q}.
\end{align*} 
This completes the proof.
\end{proof}

Now we can complete the proof of Theorem~\ref{T-9.1}
\begin{proof}[\textit{{\bf Proof of Theorem~\ref{T-9.1}}}]
Assume that \hyperlink{11}{(1)} holds. Since $L^q_w((0, R), L^p_v(\cX))$ is a Bochner space, functions of the form $\sum_{i=1}^k \zeta_i(t)\tilde{\zeta}_i(x)$
are dense. We can even choose $\zeta_i\in C_c((0, R))$ and $\tilde\zeta_i$ in $C_c(\cX)$. By Assumption~\ref{Assump-9.1}(i), 
$$\int_{\cX}\varphi_t(x, y) \tilde\zeta_i(y) \, d\mu(y)\to \tilde\zeta_i(x)\quad \text{for all}\; x\in\cX,$$
as $t\to 0$. Hence 
$$\varphi_t\odot (\sum_{i=1}^k \zeta_i \tilde{\zeta}_i)(x,t)\to 0\quad \text{as}\; t\to 0$$
for all $x$. Following the proof of Theorem~\ref{T-1.3} we see that \hyperlink{12}{(2)} holds. It is easy to see that
$\hyperlink{12}{(2)}\Rightarrow \hyperlink{13}{(3)}$ and $\hyperlink{13}{(3)}\Rightarrow \hyperlink{14}{(4)}$.
Now assume \hyperlink{14}{(4)} holds. From the definition of $D_{q, p}$ and duality, there exists $t_0\in (0, T)$ such that
$\varphi_t\odot F(\cdot, t_0)$ maps $L^q_w((0, t_0), L^p_v(\cX))$ into $L^p_{u_2}(\cX)$ for some weight function $u_2$. Setting
$R=t_0/\gamma$, we obtain from Lemma~\ref{L9.3} that $\mathcal{M}_{R, \Gamma(R)}$ maps $L^q_w((0, t_0), L^p_v(\cX))$
into $L^p_{u_1}(\cX)$ for some weight function $u_1$. Letting $u=\min\{(\xi_1(\cdot, R_1))^{-p}u_1, 
(\xi_2(\cdot, R_1))^{-p}u_2\},$ we complete the proof.
\end{proof}

\subsection{An application to  the classical  nonlinear  heat equation}
In this section, we showcase an interesting application of Theorems~\ref{T-1.3} and ~\ref{T-9.1} in the context of classical heat equation.
More specifically, consider the following initial value nonhomogeneous problem  
\begin{equation}\label{eq:nhhom}
\begin{split}
    u_{t}(x, t) - \Delta u(x, t) &= F(x, t) \quad \text{in}\; \Rn\times (0, \infty),
    \\
    u(x,0) &= g(x) \quad \text{in}\; \Rn.
\end{split}
\end{equation}  
Using Duhamel principle, the solution of \eqref{eq:nhhom} is written as 
\[
u(x, t) = \int_{\mathbb{R}^n} \hat{p}_{t}(x-y) g(y) \, dy + \int_{0}^{t} \int_{\mathbb{R}^n} \hat{p}_{t-s}(x-y) F(y, s) \, dy \, ds,
\] 
where $\hat{p}_t(x)= (4\pi t)^{-\frac{n}{2}}e^{-\frac{|x|^2}{4t}}.$ If for some weights $\tilde v\in D_{p_1}$, and $w v\in D_{q, p_2}, \,p_2\in (1, \infty),$ we have \(g \in L_{\tilde v}^{p_1}(\mathbb{R}^n)\) and
\(F \in L_{w}^{q}((0, T), L_{v}^{p_2}(\mathbb{R}^n))\) then 
$\lim_{t \to 0} u(x, t) = g(x)$ almost surely. We can even consider nonlinear nonhomogeneous term and find a suitable 
space where such pointwise attainability of the initial data holds. Towards this goal, we
consider the following nonlinear partial differential equation
\begin{equation}\label{E-nonheat}
\begin{split}
\partial_t u - \Delta u &= |u|^{\alpha - 1} u \quad \text{in}\; \Rn\times (0, \infty),
\\
u(0, x) & = U(x)\quad  x \in \mathbb{R}^n
\end{split}
\end{equation}
for $\alpha> 1$. By a solution of \eqref{E-nonheat} we mean a function $u:\Rn\times(0, \infty)\to\mathbb{R}$ that satisfies
$$u(x, t) = \hat{p}_t \ast U(x) + \int_0^t \int_{\mathbb{R}^n} \hat{p}_{t-s}(x-y) N(u(y, s)) \, dy \, ds,$$
where \( N(u) := |u|^{\alpha - 1} u \). Let us set the weight functions \( v(x) = 1 \), \( w(t) = t^k \) with \( k > 0 \), and the space \( X = L_w^q\left((0, T), L^p(\mathbb{R}^n)\right) \).
Define the set 
\[
M_A = \left\{ u \in X \mid \|u\|_X \leq A \right\}.
\]
Also, define the map
\[
\mathcal{T}(u) = \hat{p}_t \ast U + \int_0^t \int_{\mathbb{R}^n} \hat{p}_{t-s}(x-y) N(u(y, s)) \, dy \, ds.
\]
Note that a fixed point of $\mathcal{T}$ is a solution to \eqref{E-nonheat}. We denote by
$$h(x, t)=\hat{p}_t * U(x).$$
\begin{theorem}\label{T9.4}
Suppose that $h\in L_w^q\left((0, T_0), L^p(\mathbb{R}^n)\right)$ for some $T_0>0$. Also, let the following
hold:
\begin{equation}\label{E9.5}
\begin{split}
\frac{k\alpha}{\alpha-1}<1 , \beta:=\frac{n}{2} \left(\frac{\alpha-1}{p}\right) \left(\frac{q}{q-\alpha}\right) &<1,
\; \min\{p, q\}>\alpha,
\\
k + (q-\alpha) \left(1 - \beta - \frac{k\alpha}{q-\alpha}\right) &> -1.
\end{split}
\end{equation}
Then there exists $T>0$ such that $\mathcal{T}$ has fixed point in $X$. Furthermore, given $A>0$, 
there exists $T>0$ such that any  two solutions in $M_A$ would agree in $[0, T]$.
\end{theorem}

\begin{proof}
The proof technique uses standard contraction principle. We fix 
$$A> B:=\left[\int_0^{T_0} \norm{h(\cdot, t)}^q_{L_p(\Rn)} w(t) dt\right]^{\frac{1}{q}}.$$
We find a $T>0$ such that the operator \( \mathcal{T} \) maps \( M_A \) into itself, that is, \( T: M_A \to M_A \), and \( T \) is a contraction mapping.
 Denote 
\[
Gu(x, t) = \int_0^t \int_{\mathbb{R}^n} \hat{p}_{t-s}(x-y) N(u(y, s)) \, dy \, ds.
\]
To estimate \( \|Gu(\cdot, t)\|_{L^p(\mathbb{R}^n)}  \), we use Young's convolution inequality along with the conditions \( \alpha > 1 \), \( \frac{p}{\alpha} > 1 \), \( r = \frac{p}{p + 1 - \alpha} \), and \( q > \alpha \). This gives
\[
\begin{aligned}
\|Gu(\cdot, t)\|_{L^p(\mathbb{R}^n)} 
&\leq \left[\int_{\mathbb{R}^n} \left( \int_0^t \int_{\mathbb{R}^n} \hat{p}_{t-s}(x-y) |u(y, s)|^\alpha \, dy \, ds \right)^p dx \right]^{1/p} \\
&\leq \int_0^t \left[\int_{\mathbb{R}^n} \left(\hat{p}_{t-s} * |u|^\alpha(\cdot, s)\right)^p(x) \, dx \right]^{1/p} \, ds \\
&\leq \int_0^t \|\hat{p}_{t-s}\|_{L^r(\Rn)} \|u(\cdot, s)\|_{L^p(\Rn)}^\alpha \, ds \\
&= \int_0^t \|\hat{p}_{t-s}\|_{L^r(\Rn)} \|u(\cdot, s)\|_{L^p(\Rn)}^\alpha s^{\frac{k\alpha}{q}} s^{-\frac{k\alpha}{q}} \, ds \\
&\leq \left(\int_0^t \|u(\cdot, s)\|_{L^p(\Rn)}^q s^k \, ds \right)^{\alpha/q} \left(\int_0^t \|\hat{p}_{t-s}\|_{L^r(\Rn)}^{q/(q-\alpha)} s^{-\frac{k\alpha}{q-\alpha}}\, ds \right)^{(q-\alpha)/q} \\
&\leq \|u\|_X^\alpha \left(\int_0^t \|\hat{p}_{t-s}\|_{L^r(\Rn)}^{q/(q-\alpha)} s^{-\frac{k\alpha}{q-\alpha}}\, ds \right)^{(q-\alpha)/q}.
\end{aligned}
\]
To estimate the last integration we recall  that
\[
\|\hat{p}_{t-s}\|_{L^r(\Rn)} \leq (4\pi)^{\frac{r-1}{r}}\, (t-s)^{-\frac{n}{2}\left(\frac{r-1}{r}\right)}.
\]
Thus,  
\[
\|\hat{p}_{t-s}\|_{L^r(\Rn)}^{q/(q-\alpha)} \leq (4\pi)^{\frac{(r-1)q}{r(q-\alpha)}}\, (t-s)^{-\frac{n}{2}\left(\frac{r-1}{r}\right)\left(\frac{q}{q-\alpha}\right)}.
\]
Let us denote by
\[
\beta = \frac{n}{2} \left(\frac{r-1}{r}\right) \left(\frac{q}{q-\alpha}\right).
\]  
We then estimate
\[
\int_0^t \|\hat{p}_{t-s}\|_{L^r(\Rn)}^{q/(q-\alpha)} s^{-\frac{k\alpha}{q-\alpha}} ds 
\leq  (4\pi)^{\frac{(r-1)q}{r(q-\alpha)}}\, \int_0^t (t-s)^{-\beta} s^{-\frac{k\alpha}{q-\alpha}} ds.
\]
By substituting \( s = tl \), we rewrite the integral as
\[
\int_0^t (t-s)^{-\beta} s^{-\frac{k\alpha}{q-\alpha}} ds 
= t^{1 - \beta - \frac{k\alpha}{q-\alpha}} \int_0^1 (1-l)^{-\beta} l^{-\frac{k\alpha}{q-\alpha}} dl.
\]
By our assumption we have \( \beta < 1 \) and \( \frac{k\alpha}{q-\alpha} < 1 \) which make the above integral finite. Consequently,  
\[
\left(\int_0^t \|\hat{p}_{t-s}\|_{L^r(\Rn)}^{q/(q-\alpha)} s^{-\frac{k\alpha}{q-\alpha}}\, ds\right)^{(q-\alpha)/q} 
\leq \kappa_1 t^{\frac{q-\alpha}{q} \left(1 - \beta - \frac{k\alpha}{q-\alpha}\right)},
\]
for some universal constant $\kappa_1$.
Thus, the estimate for \( \|Gu\|_X \) becomes  
\[
\left(\int_0^T \|Gu\|^q\right)^{1/q} \leq \kappa_1 \|u\|_X^\alpha \left( \int_0^T t^{k + (q-\alpha) \left(1 - \beta - \frac{k\alpha}{q-\alpha}\right)} dt \right)^{1/q}.
\]
Recall that 
\[
k + (q-\alpha) \left(1 - \beta - \frac{k\alpha}{q-\alpha}\right) > -1.
\]
So, if we choose $T\leq T_0$ small enough so that
$$B + \kappa_1 A^\alpha \left( \int_0^T t^{k + (q-\alpha) \left(1 - \beta - \frac{k\alpha}{q-\alpha}\right)} dt \right)^{1/q}<A,$$
we get 
\[
\|\mathcal{T}(u)\|_X \leq A.
\]
Thus, $\mathcal{T}:M_A\mapsto M_A$. Next, we show the contraction property of $\mathcal{T}$ in $M_A$.
 Applying Young's convolution inequality, we note that
\[
\begin{aligned}
\|\mathcal{T}(u_1)(\cdot, t) - \mathcal{T}(u_2)(\cdot, t)\|_{L^p(\mathbb{R}^n)} 
&= \|Gu_1(\cdot, t) - Gu_2(\cdot, t)\|_{L^p(\mathbb{R}^n)} \\
&= \left[ \int_{\mathbb{R}^n} \left( \int_0^t \int_{\mathbb{R}^n} \hat{p}_{t-s}(x-y) \left|u_1(y, s)^\alpha - u_2(y, s)^\alpha \right| \, dy \, ds \right)^p dx \right]^{1/p} \\
&\leq \int_0^t \|\hat{p}_{t-s}\|_{L^r(\mathbb{R}^n)} \|u_1^\alpha(\cdot, s) - u_2^\alpha(\cdot, s)\|_{L^{p/\alpha}(\mathbb{R}^n)} \, ds.
\end{aligned}
\]
Next, employing the inequality
\[
\left| |u_1|^{\alpha-1} u_1 - |u_2|^{\alpha-1} u_2 \right| \leq \alpha \left( |u_1|^{\alpha-1} + |u_2|^{\alpha-1} \right) |u_1 - u_2|,
\]
together with the H\"{o}lder's inequality, we obtain
\[
\left\||u_1|^{\alpha-1} u_1 - |u_2|^{\alpha-1} u_2 \right\|_{L^{p/\alpha}(\mathbb{R}^n)} 
\leq \alpha g \|u_1 - u_2\|_{L^p(\mathbb{R}^n)},
\]
where $g(s)= \|u_1(.,s)\|_{L^p(\mathbb{R}^n)}^{\alpha-1} + \|u_2(.,s)\|_{L^p(\mathbb{R}^n)}^{\alpha-1}.$ Substituting this bound into the earlier estimate, it follows that
\[
\begin{aligned}
&\|\mathcal{T}(u_1)(\cdot, t) - \mathcal{T}(u_2)(\cdot, t)\|_{L^p(\mathbb{R}^n)}
\\
&\leq \alpha \int_0^t \|\hat{p}_{t-s}\|_{L^r(\mathbb{R}^n)} \, g(s) \|u_1(\cdot, s) - u_2(\cdot, s)\|_{L^p(\mathbb{R}^n)} \, ds\\
&\leq \alpha \int_0^t \|\hat{p}_{t-s}\|_{L^r(\mathbb{R}^n)} \, g(s) \|u_1(\cdot, s) - u_2(\cdot, s)\|_{L^p(\mathbb{R}^n)} s^{\frac{k}{q}}s^{-\frac{k}{q}}\, ds\\
&\leq \alpha\|u_1-u_2\|_X\left(\int_0^t  \|\hat{p}_{t-s}\|_{L^r(\mathbb{R}^n)}^\frac{q}{(q-1)} \left( g(s) \right)^\frac{q}{(q-1)}  s^{-\frac{k}{(q-1)}}  \,ds       \right)^\frac{(q-1)}{q}\\
&= \alpha\|u_1-u_2\|_X\left(\int_0^t  \|\hat{p}_{t-s}\|_{L^r(\mathbb{R}^n)}^\frac{q}{(q-1)} \left(\, g(s)s^\frac{k(\alpha-1)}{q}\right)^\frac{q}{(q-1)}  s^{-\frac{k\alpha}{(q-1)}}  \,ds       \right)^\frac{(q-1)}{q}\\
&\leq \alpha\|u_1-u_2\|_X\left(\int_0^t  \left(\, g(s)s^\frac{k(\alpha-1)}{q}\right)^\frac{q}{(\alpha-1)}  \,ds       \right)^\frac{(\alpha-1)}{q} \left(\int_0^t \|\hat{p}_{t-s}\|_{L^r(\mathbb{R}^n)}^\frac{q}{(q-\alpha)}  s^{-\frac{k\alpha}{(q-\alpha)}} \,ds       \right)^\frac{(q-\alpha)}{q}\\
&\leq \alpha\|u_1-u_2\|_X\left(\|u_1\|_X+\|u_2\|_X\right)^{(\alpha-1)}  \left(\int_0^t \|\hat{p}_{t-s}\|_{L^r(\mathbb{R}^n)}^\frac{q}{(q-\alpha)}  s^{-\frac{k\alpha}{(q-\alpha)}} \,ds       \right)^\frac{(q-\alpha)}{q}
\\
&\leq \alpha\|u_1-u_2\|_X (2A)^{(\alpha-1)}  \left(\int_0^t \|\hat{p}_{t-s}\|_{L^r(\mathbb{R}^n)}^\frac{q}{(q-\alpha)}  s^{-\frac{k\alpha}{(q-\alpha)}} \,ds       \right)^\frac{(q-\alpha)}{q}.
\end{aligned}
\]
Now the last integral can be computed as before, and choosing $T$ small, if required,  we can ensure that 
\( \|\mathcal{T}(u_1)-\mathcal{T}(u_2)\|_X < \varrho \|u_1-u_2\|_X\), for some $\varrho\in (0,1)$. This gives the existence of a unique
fixed point in $M_A$. The second part follows by observing that given any $A>0$, there exists $T>0$  such that
$\mathcal{T}$ becomes a contraction in $M_A$, leading to a unique fixed point.
\end{proof}

Combining Theorems~\ref{T-9.1} and ~\ref{T9.4} we have the following convergence result for the power type nonlinearity.
\begin{theorem}
Suppose that \eqref{E9.5} holds together with $k<q-1$. Also, let $\hat{p}\ast U\in X$ and $U\in L^p_{\tilde v}(\mathbb{R}^n)$ for some weight function
$\tilde{v}\in D_p$.
Then the solution of \eqref{E-nonheat}
satisfies $\lim_{t\to 0} u(x, t)=U(x)$ almost surely.
\end{theorem}

\begin{proof}
Existence of a unique solution is guaranteed by Theorem~\ref{T9.4}. Since $k<q-1$, we have for any $t_0>0, p>1$ that
$$
\int_0^{t_0}  \norm{\hat{p}_{t_0-s}(x-\cdot)}^{q'}_{L^{p'}(\mathbb{R}^n)} w^{-\frac{q'}{q}}(s)\, ds\lesssim 
\int_0^{t_0} (t_0-s)^{\frac{q'}{p}} s^{-\frac{k}{q-1}}\, ds<\infty.
$$
Hence the result follows from Theorems~\ref{T-1.3} and ~\ref{T-9.1}.

\end{proof}

\subsection*{Acknowledgement} This research of Anup Biswas was supported in part by a SwarnaJayanti fellowship SB/SJF/2020-21/03. The third author gratefully acknowledges financial support from SERB, Government of India (Project Code: 30120523), through A.B., and sincerely acknowledges the support provided by IISER Pune, Government of India.

\bigskip


\bigskip

\bibliographystyle{plain}
\bibliography{bibliography.bib}

\begin{thebibliography}{10}

\bibitem{stingaPA}
Ibraheem Abu-Falahah, Pablo~Ra\'{u}l Stinga, and Jos\'{e}~L. Torrea.
\newblock A note on the almost everywhere convergence to initial data for some
  evolution equations.
\newblock {\em Potential Anal.}, 40(2):195--202, 2014.

\bibitem{Ali-Bor}
Charalambos~D. Aliprantis and Kim~C. Border.
\newblock {\em Infinite dimensional analysis}.
\newblock Springer, Berlin, third edition, 2006.
\newblock A hitchhiker's guide.

\bibitem{ARBB}
I.~Alvarez-Romero, B.~Barrios, and J.~J. Betancor.
\newblock Pointwise convergence of the heat and subordinates of the heat
  semigroups associated with the {L}aplace operator on homogeneous trees and
  two weighted {Lp} maximal inequalities.
\newblock {\em Communications in Contemporary Mathematics}, 27(02):2450010,
  2025.

\bibitem{Amri-Hammi}
B\'echir Amri and Amel Hammi.
\newblock Dunkl-{S}chr\"odinger operators.
\newblock {\em Complex Anal. Oper. Theory}, 13(3):1033--1058, 2019.

\bibitem{Amri-Sifi}
B\'echir Amri and Mohamed Sifi.
\newblock Singular integral operators in {D}unkl setting.
\newblock {\em J. Lie Theory}, 22(3):723--739, 2012.

\bibitem{Aronson}
D.~G. Aronson.
\newblock Non-negative solutions of linear parabolic equations.
\newblock {\em Ann. Scuola Norm. Sup. Pisa Cl. Sci. (3)}, 22:607--694, 1968.

\bibitem{Avram}
Parthena Avramidou.
\newblock Convolution operators induced by approximate identities and pointwise
  convergence in {$L_p({\Bbb R})$} spaces.
\newblock {\em Proc. Amer. Math. Soc.}, 133(1):175--184, 2005.

\bibitem{Barlow}
Martin~T. Barlow.
\newblock Diffusions on fractals.
\newblock In {\em Lectures on probability theory and statistics
  ({S}aint-{F}lour, 1995)}, volume 1690 of {\em Lecture Notes in Math.}, pages
  1--121. Springer, Berlin, 1998.

\bibitem{BB99}
Martin~T. Barlow and Richard~F. Bass.
\newblock Brownian motion and harmonic analysis on {S}ierpinski carpets.
\newblock {\em Canad. J. Math.}, 51(4):673--744, 1999.

\bibitem{BP88}
Martin~T. Barlow and Edwin~A. Perkins.
\newblock Brownian motion on the {S}ierpi\'nski gasket.
\newblock {\em Probab. Theory Related Fields}, 79(4):543--623, 1988.

\bibitem{DR1}
Divyang~G. {Bhimani} and Rupak~K. {Dalai}.
\newblock {Pointwise convergence for the heat equation on tori $\mathbb T^n$
  and waveguide manifold $\mathbb T^n \times \mathbb R^m$}.
\newblock {\em arXiv e-prints}, page arXiv:2406.14271, June 2024.

\bibitem{Bogdan}
Krzysztof Bogdan, Tomasz Grzywny, and Micha{\l} Ryznar.
\newblock Density and tails of unimodal convolution semigroups.
\newblock {\em J. Funct. Anal.}, 266(6):3543--3571, 2014.

\bibitem{Bo16}
J.~Bourgain.
\newblock A note on the {S}chr\"odinger maximal function.
\newblock {\em J. Anal. Math.}, 130:393--396, 2016.

\bibitem{brezis1996nonlinear}
Ha{\"\i}m Brezis and Thierry Cazenave.
\newblock A nonlinear heat equation with singular initial data.
\newblock {\em Journal D'Analyse Math{\'e}matique}, 68:277--304, 1996.

\bibitem{BrunoARXIV}
Tommaso {Bruno} and Effie {Papageorgiou}.
\newblock {Pointwise convergence to initial data for some evolution equations
  on symmetric spaces}.
\newblock {\em J. Anal. Math}, to appear, 2023.

\bibitem{BCM10}
A.~Buades, B.~Coll, and J.~M. Morel.
\newblock Image denoising methods. {A} new nonlocal principle.
\newblock {\em SIAM Rev.}, 52(1):113--147, 2010.
\newblock Reprint of ``A review of image denoising algorithms, with a new one''
  [MR2162865].

\bibitem{CardosoJEE}
Isolda Cardoso.
\newblock On the pointwise convergence to initial data of heat and {P}oisson
  problems for the {B}essel operator.
\newblock {\em J. Evol. Equ.}, 17(3):953--977, 2017.

\bibitem{CardosoARXIV}
Isolda {Cardoso}.
\newblock {About the convergence to initial data of the heat problem on the
  Heisenberg group}.
\newblock {\em arXiv e-prints}, page arXiv:2309.08785, September 2023.

\bibitem{Car80}
Lennart Carleson.
\newblock Some analytic problems related to statistical mechanics.
\newblock In {\em Euclidean harmonic analysis ({P}roc. {S}em., {U}niv.
  {M}aryland, {C}ollege {P}ark, {M}d., 1979)}, volume 779 of {\em Lecture Notes
  in Math.}, pages 5--45. Springer, Berlin, 1980.

\bibitem{Carleson81}
Lennart Carleson and P~Jones.
\newblock Weighted norm inequalities and a theorem of koosis.
\newblock {\em Mittag-Leffler Inst., report}, (2), 1981.

\bibitem{MR499948}
Ronald~R. Coifman and Guido Weiss.
\newblock {\em Analyse harmonique non-commutative sur certains espaces
  homog\`enes}, volume Vol. 242 of {\em Lecture Notes in Mathematics}.
\newblock Springer-Verlag, Berlin-New York, 1971.
\newblock \'Etude de certaines int\'egrales singuli\`eres.

\bibitem{MR447954}
Ronald~R. Coifman and Guido Weiss.
\newblock Extensions of {H}ardy spaces and their use in analysis.
\newblock {\em Bull. Amer. Math. Soc.}, 83(4):569--645, 1977.

\bibitem{DK82}
Bj\"orn E.~J. Dahlberg and Carlos~E. Kenig.
\newblock A note on the almost everywhere behavior of solutions to the
  {S}chr\"odinger equation.
\newblock In {\em Harmonic analysis ({M}inneapolis, {M}inn., 1981)}, volume 908
  of {\em Lecture Notes in Math.}, pages 205--209. Springer, Berlin-New York,
  1982.

\bibitem{Dai-Xu}
Feng Dai and Yuan Xu.
\newblock {\em Analysis on {$h$}-harmonics and {D}unkl transforms}.
\newblock Advanced Courses in Mathematics. CRM Barcelona.
  Birkh\"auser/Springer, Basel, 2015.

\bibitem{Daub}
Ingrid Daubechies.
\newblock An uncertainty principle for fermions with generalized kinetic
  energy.
\newblock {\em Comm. Math. Phys.}, 90(4):511--520, 1983.

\bibitem{Guzman}
Miguel de~Guzm\'{a}n.
\newblock {\em Real variable methods in {F}ourier analysis}.
\newblock Notas de Matem\'{a}tica. [Mathematical Notes]. North-Holland
  Publishing Co., Amsterdam-New York, 1981.
\newblock North-Holland Mathematics Studies, 46.

\bibitem{Deleaval}
Luc Deleaval.
\newblock A note on the behavior of the {D}unkl maximal operator.
\newblock {\em Adv. Pure Appl. Math.}, 9(4):237--246, 2018.

\bibitem{DFV14}
Serena Dipierro, Alessio Figalli, and Enrico Valdinoci.
\newblock Strongly nonlocal dislocation dynamics in crystals.
\newblock {\em Comm. Partial Differential Equations}, 39(12):2351--2387, 2014.

\bibitem{DG23}
Sean {Douglas} and Loukas Grafakos.
\newblock {Remarks on almost everywhere convergence and approximate
  identities}.
\newblock {\em Acta Mathematica Sinica}, (accepted), 2023.

\bibitem{DGL17}
Xiumin Du, Larry Guth, and Xiaochun Li.
\newblock A sharp {S}chr\"odinger maximal estimate in {$\Bbb R^2$}.
\newblock {\em Ann. of Math. (2)}, 186(2):607--640, 2017.

\bibitem{DZ19}
Xiumin Du and Ruixiang Zhang.
\newblock Sharp {$L^2$} estimates of the {S}chr\"odinger maximal function in
  higher dimensions.
\newblock {\em Ann. of Math. (2)}, 189(3):837--861, 2019.

\bibitem{Dunkl}
Charles~F. Dunkl.
\newblock Differential-difference operators associated to reflection groups.
\newblock {\em Trans. Amer. Math. Soc.}, 311(1):167--183, 1989.

\bibitem{Dunkl-transform}
Charles~F. Dunkl.
\newblock Hankel transforms associated to finite reflection groups.
\newblock In {\em Hypergeometric functions on domains of positivity, {J}ack
  polynomials, and applications ({T}ampa, {FL}, 1991)}, volume 138 of {\em
  Contemp. Math.}, pages 123--138. Amer. Math. Soc., Providence, RI, 1992.

\bibitem{DA23}
Jacek Dziuba\'nski and Agnieszka Hejna.
\newblock Upper and lower bounds for the {D}unkl heat kernel.
\newblock {\em Calc. Var. Partial Differential Equations}, 62(1):Paper No. 25,
  18, 2023.

\bibitem{Fefferman}
C.~Fefferman and E.~M. Stein.
\newblock Some maximal inequalities.
\newblock {\em Amer. J. Math.}, 93:107--115, 1971.

\bibitem{TorreaPM}
Luz~M. Fern\'{a}ndez-Cabrera and Jos\'{e}~L. Torrea.
\newblock Vector-valued inequalities with weights.
\newblock {\em Publ. Mat.}, 37(1):177--208, 1993.

\bibitem{Folland}
Gerald~B. Folland.
\newblock {\em Introduction to partial differential equations}.
\newblock Princeton University Press, Princeton, NJ, second edition, 1995.

\bibitem{FS92}
M.~Fukushima and T.~Shima.
\newblock On a spectral analysis for the {S}ierpi\'nski gasket.
\newblock {\em Potential Anal.}, 1(1):1--35, 1992.

\bibitem{Jose1985}
Jos\'{e} Garc\'{\i}a-Cuerva and Jos\'{e}~L. Rubio~de Francia.
\newblock {\em Weighted norm inequalities and related topics}, volume 116 of
  {\em North-Holland Mathematics Studies}.
\newblock North-Holland Publishing Co., Amsterdam, 1985.
\newblock Notas de Matem\'{a}tica [Mathematical Notes], 104.

\bibitem{Lagguare2017}
G.~Garrig\'{o}s, S.~Hartzstein, T.~Signes, and B.~Viviani.
\newblock A.e. convergence and 2-weight inequalities for {P}oisson-{L}aguerre
  semigroups.
\newblock {\em Ann. Mat. Pura Appl. (4)}, 196(5):1927--1960, 2017.

\bibitem{torreaTAMS}
Gustavo Garrig\'{o}s, Silvia Hartzstein, Teresa Signes, Jos\'{e}~Luis Torrea,
  and Beatriz Viviani.
\newblock Pointwise convergence to initial data of heat and {L}aplace
  equations.
\newblock {\em Trans. Amer. Math. Soc.}, 368(9):6575--6600, 2016.

\bibitem{Grafakos}
Loukas Grafakos, Liguang Liu, and Dachun Yang.
\newblock Vector-valued singular integrals and maximal functions on spaces of
  homogeneous type.
\newblock {\em Math. Scand.}, 104(2):296--310, 2009.

\bibitem{Gri}
Alexander Grigor'yan.
\newblock Heat kernels on weighted manifolds and applications.
\newblock In {\em The ubiquitous heat kernel}, volume 398 of {\em Contemp.
  Math.}, pages 93--191. Amer. Math. Soc., Providence, RI, 2006.

\bibitem{haraux1982non}
Alain Haraux and Fred~B Weissler.
\newblock Non-uniqueness for a semilinear initial value problem.
\newblock {\em Indiana University Mathematics Journal}, 31(2):167--189, 1982.

\bibitem{vivianiPAMS}
Silvia~I. Hartzstein, Jos\'{e}~L. Torrea, and Beatriz~E. Viviani.
\newblock A note on the convergence to initial data of heat and {P}oisson
  equations.
\newblock {\em Proc. Amer. Math. Soc.}, 141(4):1323--1333, 2013.

\bibitem{Jacob}
N.~Jacob.
\newblock {\em Pseudo differential operators and {M}arkov processes. {V}ol.
  {I}.}
\newblock Imperial College Press, London, 2001.
\newblock Fourier analysis and semigroups.

\bibitem{Kalf}
H.~Kalf, U.-W. Schmincke, J.~Walter, and R.~W\"{u}st.
\newblock On the spectral theory of {S}chr\"{o}dinger and {D}irac operators
  with strongly singular potentials.
\newblock In {\em Spectral theory and differential equations ({P}roc.
  {S}ympos., {D}undee, 1974; dedicated to {K}onrad {J}\"{o}rgens)}, Lecture
  Notes in Math., Vol. 448, pages 182--226. Springer, Berlin-New York, 1975.

\bibitem{Kerman}
R.~A. Kerman.
\newblock Pointwise convergence approximate identities of dilated radially
  decreasing kernels.
\newblock {\em Proc. Amer. Math. Soc.}, 101(1):41--44, 1987.

\bibitem{Kigami}
Jun Kigami.
\newblock {\em Analysis on fractals}, volume 143 of {\em Cambridge Tracts in
  Mathematics}.
\newblock Cambridge University Press, Cambridge, 2001.

\bibitem{Klages}
R.~Klages, G.~Radons, and I.M. Sokolov.
\newblock {\em Anomalous Transport: Foundations and Applications}.
\newblock Wiley, 2008.

\bibitem{Li-Yau}
Peter Li and Shing-Tung Yau.
\newblock On the parabolic kernel of the {S}chr\"odinger operator.
\newblock {\em Acta Math.}, 156(3-4):153--201, 1986.

\bibitem{LS10}
Elliott~H. Lieb and Robert Seiringer.
\newblock {\em The stability of matter in quantum mechanics}.
\newblock Cambridge University Press, Cambridge, 2010.

\bibitem{Hardy}
G.~Metafune, L.~Negro, and C.~Spina.
\newblock Sharp kernel estimates for elliptic operators with second-order
  discontinuous coefficients.
\newblock {\em J. Evol. Equ.}, 18(2):467--514, 2018.

\bibitem{Mittag}
G.~Mittag-Leffler.
\newblock Au lecteur.
\newblock {\em Acta Math.}, 38(1):1--2, 1921.

\bibitem{Non-tangential}
Ebner Pineda and Wilfredo Urbina~R.
\newblock Non tangential convergence for the {O}rnstein-{U}hlenbeck semigroup.
\newblock {\em Divulg. Mat.}, 16(1):107--124, 2008.

\bibitem{Philippe}
Pavol Quittner and Philippe Souplet.
\newblock {\em Superlinear parabolic problems}.
\newblock Birkh\"{a}user Advanced Texts: Basler Lehrb\"{u}cher. [Birkh\"{a}user
  Advanced Texts: Basel Textbooks]. Birkh\"{a}user/Springer, Cham, 2019.
\newblock Blow-up, global existence and steady states, Second edition of [
  MR2346798].

\bibitem{Rosler}
Margit R\"{o}sler.
\newblock Dunkl operators: theory and applications.
\newblock In {\em Orthogonal polynomials and special functions ({L}euven,
  2002)}, volume 1817 of {\em Lecture Notes in Math.}, pages 93--135. Springer,
  Berlin, 2003.

\bibitem{Rubio81}
Jos\'{e}~L. Rubio~de Francia.
\newblock Weighted norm inequalities and vector valued inequalities.
\newblock In {\em Harmonic analysis ({M}inneapolis, {M}inn., 1981)}, volume 908
  of {\em Lecture Notes in Math.}, pages 86--101. Springer, Berlin-New York,
  1982.

\bibitem{SSV}
Ren\'e{}~L. Schilling, Renming Song, and Zoran Vondra\v{c}ek.
\newblock {\em Bernstein functions}, volume~37 of {\em De Gruyter Studies in
  Mathematics}.
\newblock Walter de Gruyter \& Co., Berlin, second edition, 2012.
\newblock Theory and applications.

\bibitem{Shapiro}
Harold~S. Shapiro.
\newblock Convergence almost everywhere of convolution integrals with a
  dilation parameter.
\newblock In Walter Schempp and Karl Zeller, editors, {\em Constructive Theory
  of Functions of Several Variables}, pages 250--266, Berlin, Heidelberg, 1977.
  Springer Berlin Heidelberg.

\bibitem{MR1232192}
Elias~M. Stein.
\newblock {\em Harmonic analysis: real-variable methods, orthogonality, and
  oscillatory integrals}, volume~43 of {\em Princeton Mathematical Series}.
\newblock Princeton University Press, Princeton, NJ, 1993.
\newblock With the assistance of Timothy S. Murphy, Monographs in Harmonic
  Analysis, III.

\bibitem{Thangavelu-Xu}
Sundaram Thangavelu and Yuan Xu.
\newblock Convolution operator and maximal function for the {D}unkl transform.
\newblock {\em J. Anal. Math.}, 97:25--55, 2005.

\bibitem{JLVazquez}
Juan~Luis Vazquez and Enrike Zuazua.
\newblock The {H}ardy inequality and the asymptotic behaviour of the heat
  equation with an inverse-square potential.
\newblock {\em J. Funct. Anal.}, 173(1):103--153, 2000.

\bibitem{Laggure2014}
Pablo Viola and Beatriz Viviani.
\newblock Local maximal functions and operators associated to {L}aguerre
  expansions.
\newblock {\em Tohoku Math. J. (2)}, 66(2):155--169, 2014.

\bibitem{Watanabe}
Toshiro Watanabe.
\newblock The isoperimetric inequality for isotropic unimodal {L}\'evy
  processes.
\newblock {\em Z. Wahrsch. Verw. Gebiete}, 63(4):487--499, 1983.

\bibitem{Weissler81}
Fred~B. Weissler.
\newblock Existence and nonexistence of global solutions for a semilinear heat
  equation.
\newblock {\em Israel J. Math.}, 38(1-2):29--40, 1981.

\bibitem{Zhang}
Qi~S. Zhang.
\newblock Gaussian bounds for the fundamental solutions of {$\nabla (A\nabla
  u)+B\nabla u-u_t=0$}.
\newblock {\em Manuscripta Math.}, 93(3):381--390, 1997.

\end{thebibliography}
\end{document}